\documentclass[11pt]{article}
\usepackage[latin1]{inputenc}
\usepackage{latexsym}

\usepackage{graphicx}
\usepackage{epsfig}
\usepackage{amssymb}
\usepackage{amsmath}
\usepackage{amsfonts}
\usepackage{a4}
\usepackage[latin1]{inputenc}
\usepackage[english]{babel}   
\usepackage{fancyhdr}
\usepackage{lastpage}

\usepackage{rotating}
\usepackage{tabularx}
\usepackage{algorithm}

\def\R{{\sl I\kern-.27em R}}
\def\N{{\sl I\kern-.27em N}}
\def\P{{\sl I\kern-.27em P}}
\def\b{{\sl I\kern-.27em b}}

\def\R{\mathbb{R}}
\def\N{\mathbb{N}}
\def\P{\mathbb{P}}

\def\max{\mathop{\rm max}\nolimits}

\def\q{ \quad}
\def\dis{ \displaystyle}

\def\gradV{\nabla}   
\def\laplac{\Delta}  
\def\uu{{\bf u}^{m+1}}
\def\umedio{\widetilde{\bf u}^{m+1}}
\def\ee{{\bf e}^{m+1}}
\def\emedio{\widetilde{\bf e}^{m+1}}

\def\deltaumedio{\delta_t  \widetilde{\bf u}^{m+1}}

\def\deltaee{\delta_t {\bf e}^{m+1}}
\def\deltaemedio{\delta_t  {\bf e}^{m+1/2}}

\def\utilde{\widetilde{{\bf u}}^{m+1}}

\def\etilde{\widetilde{{\bf e}}^{m+1}}

\newtheorem{theorem}{Theorem}
\newtheorem{lemma}[theorem]{Lemma}
\newtheorem{corollary}[theorem]{Corollary}
\newtheorem{proposition}[theorem]{Proposition}
\newtheorem{remark}[theorem]{Remark}

\newcommand{\cqfd}{\mbox{}\nolinebreak\hfill\rule{2mm}{2mm}\medskip\par}

\newenvironment{proof}[1] {\par\noindent{\bf Proof. }{#1}}{\cqfd}
\title{Optimal first-order error estimates of a fully segregation scheme
 for the Navier-Stokes equations\thanks{The authors have been partially supported by MINECO (Spain),  Grant  MTM2012--32325
 and the second author is also partially supported  by the research group
FQM-315 of  Junta de Andaluc\'{\i}a.}}
\author{F.~Guill\'en-Gonz\'alez\footnote{Departamento de
Ecuaciones Diferenciales y An\'alisis Num\'erico and IMUS. Universidad de
Sevilla. Aptdo 1160, 41080 Sevilla (Spain), email: guillen@us.es,
fax: ++ 34 5 4552898, phone: ++ 34 5 4559907.},
  M.V.~Redondo-Neble\footnote{Departamento de Matem\'aticas. Universidad de
C\'adiz. C.A.S.E.M. Pol\'{\i}gono R\'{\i}o San Pedro S/N, $11510$
Puerto Real. C\'adiz (Spain), email: victoria.redondo@uca.es,
 phone: ++ 34 5 6016058.} }

\begin{document}

\maketitle

\begin{abstract}
A first-order linear fully discrete scheme is studied for the incompressible
time-dependent Navier-Stokes equations in three-dimensional
domains. This scheme, based on an incremental pressure
projection method,  decouples each component  of the 
velocity and the pressure, solving in each time step, a linear 
convection-diffusion problem for each component of the velocity
and a Poisson-Neumann problem for the pressure.

 Using first-order \emph{inf-sup} stable $C^0$-finite elements,
 optimal error estimates of order $O(k+h)$ are deduced without
 imposing constraints on  $h$ and $k$, the mesh size and the time step, respectively. 
 
 Finally, some numerical results are presented according the  theoretical analysis, and  also comparing to other current first-order segregated schemes. 
 
 \end{abstract}
 \medskip


\section*{Introduction}
Let us  consider the Navier-Stokes system, associated to the dynamics
of viscous and incompressible fluids filling a bounded domain
$\Omega\subset \R^3$ in a time interval $(0,T)$:
 $$
\left \{ \begin{array}{rcll}
  {\bf u}_t \,+\, ( {\bf u}\cdot \gradV ) {\bf u}\,
- \nu \, \Delta {\bf u}+\,
               \gradV\, p \,
              & =&
   \, {\bf f} & \hbox{ in $ \Omega \times (0,T)$,} \\
 \noalign{\smallskip}
 \nabla \cdot {\bf u} &=&0  & \hbox{ in $ \Omega \times (0,T)$,}  \\
\noalign{\smallskip}
{\bf u} &=& {\bf 0}  & \hbox{ on $ \partial\Omega \times (0,T)$}, \\
\noalign{\smallskip} {\bf u}_{|t=0}&=&{\bf u}_0  & \hbox{ in
$\Omega $}.
\end{array}\right.
\leqno{(P)}
 $$
where the unknowns are ${\bf u}: ({\bf x},t)\in \Omega \times
(0,T)\to \R^3$ the velocity field and  $ p:({\bf x},t)\in \Omega
\times (0,T)\to \R $ the  pressure, and data are $\nu>0$ the
  viscosity coefficient (which is assumed constant for simplicity) and
  ${\bf f}:\Omega \times (0,T)\to \R^3$ the external
  forces. We denote by   $\nabla $ the gradient operator
  and  $\laplac$ the Laplace operator.

We consider  a (uniform) partition of $[0,T]$ related to a fixed time step $k=T/M$:
$t_0=0,t_1=k,\dots , t_m=mk,\dots , t_M=T$. If $u=(u^m)_{m=0}^M$
is a given vector with $u^m\in X$ (a Banach space), let us to
introduce the following notation for discrete in time norms:
$$
\| u \|_{l^2(X)}=\left(k \sum_{m=0}^M \| u^m \|_X^2 \right)^{1/2}
\quad \hbox{and} \quad \| u \|_{l^\infty(X)}= \max_{m=0,\dots ,M} \|
u^m \|_X
$$
For simplicity, we will denote $H^1=H^1(\Omega)$ etc.,
$L^2(H^1)=L^2(0,T;H^1)$ etc., and ${\bf H^1}=H^1(\Omega)^3$ etc.  We will denote by $C>0$ different constants, always
independent of discrete parameters $k$ and $h$.

The numerical analysis for the Navier-Stokes problem $(P)$ has
received much attention in the last decades and many numerical
schemes are now available. The main (numerical) difficulties are: the coupling between the pressure term $\nabla p$ and the
incompressibility condition $\nabla\cdot{\bf u}=0$ and the
nonlinearity given by the convective terms $( {\bf u}\cdot \gradV ) {\bf
u}$.

Fractional-step projection methods are becoming widely used,
splitting the different operators appearing in the
problem.
The origin of these methods is generally credited to the works of
Chorin \cite{cho} and Temam \cite{teman}. They developed the well
known \emph{Chorin-Temam projection} method, which is a two-step
scheme,
computing firstly an intermediate velocity via a convection-diffusion
problem and secondly a  velocity-pressure pair via a divergence-free $L^2(\Omega)$-projection
problem. Afterwards, a modified
projection scheme (called \emph{incremental-pressure or Van-Kan
scheme}) was
developed \cite{vankan}, adding an explicit pressure term in the first
step and a pressure correction term  in the projection step. The main drawbacks of projection methods are that the
 end-of-step  velocity   does not satisfy the
 exact boundary conditions and the discrete pressure satisfies an 
 ``artificial'' Neumann boundary condition.

 Some current variants of projection methods are: rotational
pressure-correction schemes (\cite{t-m-v}, \cite{guermond-shen},
\cite{guermond-minev-shen}), velocity-correction schemes
(\cite{guermond-shen4},
\cite{guermond-shen3}),    consistent-splitting  schemes (\cite{guermond-shen-consistent},\cite{guermond-minev-shen}, \cite{shen-yang}) 
and  penalty pressure-projection schemes (\cite{jobelin}, \cite{angot-1},
\cite{angot-2}).   Other variants  can be seen in \cite{prohl} and \cite{pyo}.

The convergence of the \emph{Chorin-Temam projection} method was
proved first in \cite{teman2} for the time discrete scheme and afterwards  in
\cite{cho2} for a fully discrete finite element (FE) scheme.

 On the other hand,  error estimates for projection
methods were obtained (see \cite{shen}, \cite{shen2} for time
discrete schemes and \cite{gq} for a fully  discrete FE scheme).
 Basically,  the  Chorin-Temam  scheme has order $O(k^{1/2})$ in
$ l^{\infty}({\bf L}^2)\cap l^2({\bf H}^1)$ and  $O(k)$ in
$l^2({\bf L}^2)$ for the  velocity, and  $O(k^{1/2})$ in
$l^2(L^2)$ for the pressure. For the incremental-pressure scheme,
 these error estimates are improved in \cite{shen} and \cite{shen2} to
order $O(k)$ in $ l^{\infty}({\bf L}^2) \cap l^2({\bf H}^1)$ for the
 velocity and $O(k)$ in $l^2(L^2)$ for the
pressure (although this last estimate is proved only for the linear problem). In fact, these optimal
error estimates are extended  in
\cite{gq} to a fully discrete FE-stable scheme (see (\ref{cinfsup}) below)  under the constraint $k^2\le C\, h  $ in
 $3D$ domains or $k^2 \le \alpha \, (1+ \log (h^{-1}) ) $
 in $2D$ ones. The
 argument done in \cite{gq} is based on the direct comparison between an appropriate
spatial interpolation of the exact solution and the fully discrete
 scheme.

By the contrary, in this paper, we will obtain optimal
error estimates without imposing restrictions on $h$ and $k$  for a FE decoupled scheme different from scheme studied in \cite{gq} (which was not  decoupled because
the projection step is solved by means of a mixed velocity-pressure
formulation).  The argument used now is also different from \cite{gq}, because   the corresponding time
discrete scheme will be introduced as an intermediate problem. This 
argument has already been used in \cite{g-re-cras,g-re-NM,g-re-space} for a different splitting scheme (with decomposition of viscosity) applied to  Navier-Stokes
 equations. 

The particular property that some projection methods (without and
with  incremental pressure) can be rewritten as  segregated methods (decoupling  velocity and
pressure), was observed in \cite{rannacher,shen}.  
For a segregated fully discrete FE  scheme based on 
 the non-incremental projection method,
 the convergence and sub-optimal error estimates $O(k^{1/2}+h)$
 for the pressure  have been obtained in \cite{badia-codina}, without
imposing inf-sup condition, but
under the double constraint $\alpha \, h^2 \le  k \le \beta
\, h^2$.

 In this paper, we obtain optimal order $O(k+h)$ for the velocity and
  pressure, without imposing constraints on $h$ and $k$, for a time segregated scheme with first-order inf-sup stable FE spaces. 
   Up to  our knowledge, optimal first order  for the pressure  of a fully segregated scheme for the Navier-Stokes problem have not been proved before.

Ideas of this paper are being used  to design a segregated second order in time  scheme (\cite{cedya2013}). 

This paper is organized as follows:

 In Section 1, we study the time discrete scheme (see Algorithm~\ref{alg-time} below). Firstly, the  stability of this  scheme  is deduced,
 and we introduce the  discrete in time problems satisfied by
  errors and the regularity hypotheses that must be imposed on the exact
solution. Afterwards, we obtain $O(k)$ accuracy for
the velocity in $l^{\infty}({\bf L}^2) \cap l^2({\bf H}^1)$. As a consequence, the 
 velocity is bounded in $l^{\infty}( {\bf H^1})$. Then, we deduce $O(k)$ for the discrete in time derivative of velocities in $l^{\infty}({\bf L}^2) \cap l^2({\bf H}^1)$. Finally,
$O(k)$  for the velocity in $l^{\infty}({\bf H}^1)$ and for
the pressure in $l^{\infty}(L^2)$ hold.

  Section 2 is devoted to study the fully discrete FE scheme (see Algorithm~\ref{algo-fully} below). 
  We present the FE-stable spaces and
their approximation properties, the fully discrete segregated scheme and the problems satisfied by the errors (comparing the time discrete Algorithm~\ref{alg-time} 
with the fully discrete Algorithm~\ref{algo-fully}).
 With respect to the spatial error estimates, firstly we obtain $O(h)$
  for the velocity  in $l^{\infty}({\bf L}^2) \cap l^2({\bf H}^1)$. Then,  
the  velocity is bounded  in $l^{\infty}( {\bf H}^1)$. 
Afterwards, by using some additional estimates for
 the time discrete scheme, $O(h)$ for the discrete in time
 derivative of  velocity  in $l^{\infty}({\bf L}^2) \cap l^2({\bf H}^1)$
 is obtained.  Finally, $O(h)$ for the velocity
in $l^{\infty}({\bf H}^1)$  and for the pressure
  in $l^{\infty}(L^2)$ are deduced.

  In Section 3, some numerical simulations are presented, showing  first order
  accuracy in time for velocity and pressure. These simulations are
  also compared with  the segregated versions of the rotational,
  consistent and penalty-projection schemes.  

 Finally,  some  conclusions are given in Section 4.

In this paper, the following discrete Gronwall's lemma will be 
used (\cite[p.~369]{hr2}):
\begin{lemma}\label{GronwallD} (Discrete Gronwall inequality) \label{GronwallD}
Let $k$, $B$ and $a_m\,$, $b_m\,$, $c_m\,$, $\gamma_m$ be
nonnegative numbers. If we assume 
$$
a_{r+1}+k\sum_{m=0}^{r}b_{m}\leq k\sum_{m=0}^{r}
  \gamma_{m}a_{m}+k\sum_{m=0}^{r}c_{m}+B\qquad\forall r\geq0,
$$
then,  one has
$$
  a_{r+1}+k\sum_{m=0}^{r}b_{m}\leq\exp\left(k\sum_{m=0}^{r}
  \gamma_{m}\right)
  \left\{k\sum_{m=0}^{r}c_{m}+B\right\}\qquad\forall r\geq 0.
$$

\end{lemma}

\section{Time discrete scheme (Algorithm~\ref{alg-time})}

The norm and inner product in $L^2 (\Omega) $ will be denoted by
$| \cdot | $ and $ \Big(\cdot, \cdot \Big) $, whereas
   the norm in $H_0^1 (\Omega) $ of the  gradient in $L^2 (\Omega) $ will be denoted
  by  $\| \cdot \| $. Any other norm in a space $X$ will be
  denoted by $\| \cdot \|_X$

Let us to  introduce  the standard Hilbert spaces in the Navier-Stokes framework:
\begin{eqnarray*}
   {\bf H}&=&\{  {\bf v} \in {\bf L}^2 (\Omega ) \ : \
  \nabla \cdot  {\bf v}  =0 \mbox{ in } \Omega,
  \   {\bf v}  \cdot {\bf n}_{\partial\Omega}=0 \},
 \\
   {\bf V}&=& \{ {\bf v} \in {\bf H}_0^1(\Omega ) \ : \
  \nabla \cdot {\bf v} =0 \mbox{ in } \Omega \},
\end{eqnarray*}
where ${\bf n}_{\partial\Omega}$ denotes the normal outwards vector to
$\partial\Omega$.

In the sequel,  the following standard skew-symmetric
form of the convective term  will be used:
$$C({\bf u},{\bf v} )=  ({\bf u} \cdot \nabla)  {\bf
  v} + \frac{1}{2} (\nabla \cdot {\bf u} )\, {\bf v}
\quad \forall \, {\bf u} \in {\bf H}_0^1,\, {\bf v}\in {\bf H}^1 ,\,
$$
 and the corresponding trilinear form
$$c({\bf u},{\bf v} ,{\bf w})=  \int_\Omega  C({\bf u},{\bf v} )\cdot {\bf
w}= \int_\Omega \Big \{ ({\bf u} \cdot \nabla ) {\bf
  v} \cdot {\bf w} + \frac{1}{2}  ( \nabla \cdot {\bf u} ) {\bf
  v}\cdot {\bf w} \Big \}, \q \ \forall \, {\bf u} \in {\bf H}_0^1,\, {\bf v}\in {\bf H}^1
,\, {\bf w} \in {\bf H}^1
$$
or equivalently
$$c({\bf u},{\bf v} ,{\bf w})= \frac{1}{2}\int_\Omega\Big\{({\bf u} \cdot \nabla ) {\bf
  v} \cdot {\bf w} - ({\bf u} \cdot \nabla ) {\bf w}\cdot {\bf v}\Big\}
  = - \int_\Omega \Big \{ ({\bf u} \cdot \nabla) {\bf
  w} \cdot {\bf v} + \frac{1}{2}  ( \nabla \cdot {\bf u} ) {\bf
  v}\cdot {\bf w} \Big \}.$$
 Previous equalities hold even in the fully discrete case, hence we can use, in the sequel, any of these three possibilities.

The trilinear form $c(\cdot , \cdot , \cdot )$ satisfies
\begin{equation}\label{antisim}
c({\bf u}, {\bf v}, {\bf v})=0 ,\qquad \forall\, {\bf u} \in {\bf
  H}_0^1 , \ \ \forall\, {\bf v} \in {\bf H}^1,
\end{equation}
$$c({\bf u}, {\bf v}, {\bf w}) \le
C \left \{
\begin{array} {l}
\| {\bf u} \| \, \| {\bf v} \|_{W^{1,3} \cap L^{\infty}} \, | {\bf w} |
\\
\| {\bf u} \|_{L^3} \, \| {\bf v} \|\, \| {\bf w} \|
\end{array} \right.
$$
where the role of ${\bf u} ,{\bf v} ,{\bf w} $ can be
interchanged, using the appropriate expression of
$c(\cdot,\cdot,\cdot)$.

For simplicity and without loss of generality, we fix the
viscosity constant $\nu =1$.

\subsection{Description of the time scheme (Algorithm~\ref{alg-time})}
Given $ ({\bf f} ^m= {\bf f}(t_m) )_{m=1}^M $, we  define an approximation $({\bf u}^m , p^m )_{m=1}^M $  of the solution $({\bf u} , p )$ of $(P)$ at time $t=t_m$, by means of  an  incremental pressure projection scheme of Van-Kan
type \cite{vankan}, splitting the nonlinearity $({\bf u}\cdot \nabla){\bf u} $ and the diffusion term $-\Delta{\bf u}$ to  the incompressibility condition $ \nabla  \cdot  {\bf u}  = 0$. Moreover, an explicit pressure term is
introduced in the convection-diffusion problem for the velocity
(Sub-step~1), with a pressure-correction in the divergence-free
projection step (Sub-step~2). See Algorithm~\ref{alg-time} for a description of the time  scheme.
\begin{algorithm}[htbp]
\begin{description}
\item[Initialization:]  Let $ p^0$ be given and to take ${\bf u}^0=\widetilde{\bf u}^0={\bf u}(0)(={\bf u}_0)$.  
\item[Step of time $m+1$:]  Let ${\bf u}^m $, $\widetilde{\bf u}^m $ and $p^m$ be given.
\begin{description}
\item[\bf Sub-step 1:]
 Find  $\utilde:\Omega \to\R^3$  solving
$$
\left \{\begin{array}{l} \displaystyle\frac{1}{k}(\utilde-{\bf
u}^{m}) + C(\widetilde {\bf u}^{m},  \utilde )
 -\laplac \utilde + \nabla p ^{m}
 =
{\bf f}^{m+1}  \q\hbox{in $\Omega$,}\\
\noalign{\smallskip}  \utilde|_{\partial\Omega} = 0 \q\hbox{on
$\partial\Omega$.}
\end{array} \right.
\leqno{(S_1)^{m+1}}
$$
\item[\bf Sub-step 2:] Find
 ${\bf u}^{m+1}:\Omega \to\R^3 $ and $p^{m+1}:\Omega \to\R$ solution of
 $$
\left \{ \begin{array}{l} \displaystyle\frac{1}{k}({\bf
u}^{m+1}-\utilde)
 +  \nabla\,( p ^{m+1}- p  ^{m}) =  0 \q\mbox{in } \Omega ,\\
\noalign{\smallskip} \nabla \cdot {\bf u}^{m+1} =0  \q\mbox{in  }
\Omega, \q \uu \cdot {\bf  n}|_{\partial \Omega}=0.
\end{array} \right.
\leqno{(S_2)^{m+1}}
$$
\end{description}
\end{description}
\caption{Time discrete algorithm} \label{alg-time}
\end{algorithm}

Notice that the  convection term has been taken in $(S_1)^{m+1}$ in the  
semi-implicit linear form $C(\widetilde {\bf u}^{m},  \utilde )$. 
On the other hand, adding $(S_1)^{m+1}$ and  $(S_2)^{m+1}$, we arrive at
$$
\left \{\begin{array}{l} \displaystyle\frac{1}{k}({\bf
u}^{m+1}-{\bf u}^{m}) +
 C(\widetilde {\bf u}^{m},  \utilde )
 -\laplac \utilde
+\nabla p^{m+1} =
{\bf f}^{m+1} \q \mbox{in  } \Omega ,\\
\noalign{\smallskip}  \quad \utilde|_{\partial \Omega }  = 0
,\quad \nabla \cdot {\bf u}^{m+1} =0 \q \mbox{in  } \Omega.
\end{array} \right.
\leqno{(S_3)^{m+1}}
$$
 In fact $(S_3)^{m+1}$ can be viewed as consistence relations, because
if $\utilde$ and ${\bf u}^{m+1}$ converge  to the same limit velocity
${\bf u}$ as $k$ goes to zero, then  taking
limits in  $(S_3)^{m+1}$, one has at least formally that ${\bf u}$ will be a solution of the exact problem $(P)$.

Now, some remarks about   Sub-step~2 are in order:
\begin{itemize}
\item Sub-step~2 can be viewed as a projection step. In fact, $\uu = P_{{\bf H}}
\utilde$ where $P_{{\bf H}} $ is the $L^2(\Omega)$-projector onto ${\bf
H}$, because $(S_2)^{m+1}$ implies in particular
$$ \Big(\uu-\utilde,{\bf u}\Big)=0 \quad \forall\, {\bf u} \in  {\bf H}.
$$
\item By using $\nabla \cdot {\bf u}^{m+1} =0$ in $
\Omega$ and  $\uu \cdot {\bf  n}|_{\partial \Omega}=0$, one has the orthogonality property 
\begin{equation}\label{L2-orthog}
    \Big( {\bf u}^{m+1} , \nabla q \Big)=0 \quad \forall\, q\in
    H^1(\Omega).
\end{equation}
\item It is well known that Sub-step 2 is equivalent to the 
following two (decoupled) problems:
\begin{enumerate}
    \item Find  $p^{m+1}:\Omega \to\R$ such that
$$\left \{
\begin{array}{c}
k\,\Delta (p^{m+1}-p^m)= \nabla \cdot
\utilde   \q\mbox{in }  \Omega  \\
\noalign{\smallskip} k\, \nabla (p^{m+1}-p^m) \cdot {\bf
n}|_{\partial \Omega} =0\q \mbox{on } \partial \Omega .
\end{array}\right.
\leqno{(S_2)_a^{m+1}}
$$
    \item  Find
 ${\bf u}^{m+1}:\Omega \to\R^3 $ as
 $$ \uu = \utilde  -k \,\nabla (p^{m+1}-p^m )  \q \mbox{in }
\Omega . \leqno{(S_2)_b^{m+1}}
$$
\end{enumerate}

\end{itemize}
\subsection{Unconditional stability and convergence
of Algorithm~\ref{alg-time}}

\begin{lemma}[Continuous dependence
  of the projection step] \label{regularidad} 
.

\noindent
 {\bf a)} (Continuous dependence with respect to ${\bf L}^2$) 
 If $\utilde$ and ${\bf u}^m \in {\bf L}^2 (\Omega )$,
 then there exists an unique
${\bf u}^{m+1} \in {\bf H}$ solution of $(S_2)^{m+1}$. Moreover,
\begin{equation}\label{continuous-dependence-a}
| \utilde |^2= | \uu | ^2 + | k\, \nabla (p^{m+1} -p^m) | ^2
\end{equation}
\begin{equation}\label{continuous-dependence-b}
|\uu - \utilde | \le | \utilde -{\bf u}^m |.
\end{equation}

\noindent
{\bf b)} (Continuous dependence with respect to ${\bf H}^1$) 
  If 
$ \utilde \in { \bf H}_0^1 ( \Omega) $ then
  ${\bf u}^{m+1}
 \in {\bf H}^1 (\Omega) \cap {\bf H} $.
 Moreover,
$$\| {\bf u}^{m+1} \| \leq C\,  \| \utilde\| .
$$
\end{lemma}
\begin{proof}

{\bf a)} 
Since $\uu = P_{{\bf H}} \utilde$, one has (\ref{continuous-dependence-a}).
%
Moreover,  estimate (\ref{continuous-dependence-b}) 
can be obtained 
directly  from the best approximation property of the $L^2$-projection:
$$\dis |\uu -\utilde |= \min_{{\bf u} \in {\bf H}}| {\bf u} -
\utilde | .$$

${\bf b)}$
 By applying the $H^2(\Omega)$-regularity of  problem
  $(S_2)_a^{m+1}$ ,  there exists a unique $p^{m+1}-p^m \in H^2 \cap
L^2_0$ 
satisfying
$$k\, \|\nabla (p^{m+1}-p^m)\|_{H^1} \leq C \|\utilde\|.$$
 Therefore,
  ${\bf u}^{m+1} \in {\bf H}^1 (\Omega)$ and
$$\|{\bf u}^{m+1}\| \leq C \Big \{\|\utilde\|
+k\, \|\nabla (p^{m+1}-p^m) \|_{H^1} \Big\}  \leq C \,\|\utilde\|.$$ 
This estimate can be understood as the $ {\bf
H}^1$-stability of the $ L^2$-projector  onto ${\bf H}$.
\end{proof}

\begin{lemma}[Stability of Algorithm~\ref{alg-time}]\label{vn}
Let ${\bf f}\in  L ^2 (0,T;{\bf H}^{-1}(\Omega))$ (${\bf
H}^{-1}(\Omega)$  being the dual space of  ${\bf
H}_0^{1}(\Omega)$) and ${\bf u} _0 \in {\bf H}$. Assuming the
following constraint on the initial discrete pressure $k\, |\nabla
p^0 | \le C_0 $, then there exists a constant $ C=C\big(C_0, {\bf
  u}_0,  {\bf
f} , \Omega \big)
> 0 $ such that,
 \begin{eqnarray*}
   &&
|\widetilde{\bf u}^{r+1}|^2 + |{\bf u}^{r+1}|^2 + | k\, \nabla
p^{r+1}|^2\leq  C ,\quad \forall \, r=0,...,M-1,
\\
&&  k \sum_{m=0}^{M-1} \left \{  \|\utilde\|^2
+ \|{\bf u}^{m+1}\|^2 \right \}
\leq C
.
\end{eqnarray*}
\end{lemma}

\begin{proof} We only give here an outline of the proof, which follows the same lines given in  the proof of Theorem \ref{v1} below.  
By making 
$$2 k\, \Big ( (S_1)^{m+1} , \umedio \Big ) +k\,  \Big (
(S_2)^{m+1} ,\umedio + \uu + k\, ( \nabla p^{m+1} +
 \nabla p^{m})  \Big ),$$
and using orthogonality property (\ref{L2-orthog}):
$$| \uu|^2+|k\, \nabla p^{m+1} |^2 -| {\bf u}^m |^2-|k\, \nabla p^m|
  ^2+ |\utilde-{\bf u}^m|^2+2\, k\, \|\utilde\| ^2=2\, k\, ({\bf
    f}^{m+1}, \utilde ),
$$
hence, by using  the discrete Gronwall's Lemma (Lemma~\ref{GronwallD}):
$$
 \| {\bf u}^{m+1} \| _{l^{\infty}({\bf L}^2)} + \| k\, \nabla p^{m+1}
 \|_{l^{\infty}( L^2)} + \| \umedio \| _{ l^2 ({\bf H}^1)}  \le C \quad
 \mbox{and} \quad \sum_{m \ge 0} |\utilde-{\bf u}^m|^2\le C .
$$
Now, accounting  Lemma \ref{regularidad}, the following
supplementary stability  estimates hold:
 $$\| \umedio  \| _{l^{\infty}({\bf L}^2)} \le C
 \q\hbox{and}\q \|  {\bf u}^{m+1} \|_{l^2 ({\bf H}^1)} \le C. $$
\end{proof}

Starting from the previous stability estimates and taking limits as $k\downarrow
0$ in $(S_3)^{m+1}$, the convergence of
the velocity approximations have already been established  (for
instance, see \cite{teman3}). Concretely,  defining ${\bf u}_{k}:(0,T]\to {\bf H}\cap {\bf H}^1(\Omega)$
as the piecewise constant functions taking the value ${\bf u}^{m+1}$
in $(t_m,t_{m+1}]$, the following result holds:
\begin{proposition}[Convergence of Algorithm~\ref{alg-time}] Under conditions of Lemma~\ref{vn}, there
exists  a subsequence $(k')$ of $(k)$, 
and a weak solution ${\bf u}\in L^{\infty} (0,T;{\bf H})\cap L^{2} (0,T;{\bf V})$  of $(P)$ in $(0,T)$, such that:
 ${\bf u}_{k'}\to {\bf u}$
  weakly-$*$ in
$L^{\infty} (0,T;{\bf H})$, weakly in $L^2 (0,T;{\bf
H}^1(\Omega)\cap {\bf H})$ and strongly in $L^2 (0,T;{\bf H})$, as $ k' \downarrow 0 $.
\end{proposition}

\subsection{Differential problems satisfied by the errors}
We will obtain  error estimates (for  velocity and pressure) with
respect to  a sufficiently regular (and unique) solution
$({\bf u}, p )$ of $(P)$. For this, we introduce the following
notations for the errors in $t=t_{m+1}$:
\begin{eqnarray*}
&& \widetilde{\bf e}^{m+1} :={\bf u}(t_{m+1})-\widetilde{\bf u}^{m+1}, \quad{\bf
e}^{m+1} :={\bf u}(t_{m+1})-{\bf u}^{m+1}, \quad e_p^{m+1} := p
(t_{m+1})- p^{m+1},
\end{eqnarray*}
and  for the discrete in time derivative of errors
$$\delta  _t  {\bf e}^{m+1}:=\displaystyle \frac{{\bf e}^{m+1}-{\bf
    e}^{m}}{k}, \quad
 \delta  _t  \widetilde{\bf e}^{m+1}:=\displaystyle \frac{\widetilde{\bf e}^{m+1}-\widetilde{\bf
     e}^{m}}{k}.$$
Subtracting    $(S_1)^{m+1}$  with the momemtum system of $(P)$ at $t=t_{m+1}$, using the integral rest and manipulating the convective terms, one has:
$$
\left \{ \begin{array}{l} \displaystyle\frac{1}{k} (\emedio -{\bf
e}^m) - \Delta \widetilde{\bf e}^{m+1}+ \nabla\, (e_{p} ^m + k\, \delta_t  p(t_{m+1}) )
=\mathcal{E}^{m+1} +{\bf NL}^{m+1} \q \mbox{in  } \Omega ,\\
\noalign{\smallskip} \widetilde{\bf e}^{m+1}|_{\partial
\Omega}={\bf 0},
\end{array}\right.
\leqno{(E_1)^ {m+1}}
$$
where
$$
\mathcal{E}^{m+1}=-\displaystyle\frac{1}{k}
\displaystyle\int_{t_m}^{t_{m+1}} (t-t_m) \, {\bf u}_{tt}  (t)\,
dt -\dis\left( \int_{t_m}^{t_{m+1}} {\bf u}_t\cdot \gradV \right)
{\bf u} (t_{m+1}) := \mathcal{E}_1^{m+1} + \mathcal{E}_2^{m+1}
$$
is the consistency error, and
$$
{\bf NL}^{m+1} = -C \Big(\widetilde{\bf e}^{m} ,  {\bf u}(t_{m+1})
\Big )-C \Big(\widetilde{\bf u}^{m}, \etilde \Big ) 
$$
are  
 terms depending  of the
convective terms. 

On the other hand,  adding and subtracting the term ${\bf
u}(t_{m+1})$ in  $(S_2)^{m+1}$,
$$
\left \{ \begin{array}{l}
\displaystyle\frac{1}{k}({\bf e}^{m+1}-\widetilde{\bf e}^{m+1})
  +  \nabla \Big (
 e_{p} ^{m+1} - e_p^m -k\,\delta_t p(t_{m+1})  \Big ) ={\bf 0} \q \mbox{in
} \Omega,
\\
 \noalign{\smallskip}
 \nabla \cdot {\bf e} ^ {m+1}  = 0 \q \mbox{in } \Omega ,
\qquad {\bf e}^{m+1} \cdot {\bf n}|_{\partial \Omega}= {\bf 0}.
\end{array}\right.
\leqno{(E_2)^{m+1}}
$$

%
%
%
%
%
Finally, adding  $(E_1)^{m+1}$ and  $(E_2)^{m+1}$, we arrive at:
$$
\left \{ \begin{array}{l}
\displaystyle\frac{1}{k}({\bf e}^{m+1}-{\bf e}^m ) -
\laplac \widetilde{\bf e}^{m+1} +\nabla e_p^{m+1}= \mathcal{E}^{m+1}
+{\bf NL}^{m+1}
  \q \mbox{in } \Omega,
\\
 \noalign{\smallskip}
\nabla \cdot {\bf e}^{m+1} = 0 \q \mbox{in } \Omega ,
\qquad {\bf e}^{m+1} \cdot {\bf n}|_{\partial \Omega}= {\bf 0}.
\end{array}\right.
\leqno{(E_3)^{m+1}}
$$
\begin{lemma}[Continuous dependence of
  the projection  errors] \label{regularidad-errores} The following inequalities hold
\begin{equation}\label{continuous-dependence-errorsa}
|\etilde|^ 2= | \ee|^2+ | k\, \nabla \, (e_p^{m+1}-e_p^m -k\, \delta_t p(t_{m+1}))|^2,
\end{equation}
$$
 | \ee -\etilde | \le | \etilde - {\bf e}^m | 
$$
$$\| {\bf e}^{m+1} \| \leq C\,  \| \etilde\| .
$$
\end{lemma}

\begin{proof} The proof is similar to  Lemma \ref{regularidad}, by
 using that $\ee= P_{{\bf H}} \etilde$.
\end{proof}

\subsection{Regularity hypotheses.}
We will assume the following regularity hypothesis
on $\Omega$:
\medskip

\noindent $ {\bf (H0)}  \quad \Omega\subset \R^3$ such that Poisson problems in $\Omega$  have $ {\bf H}^2(\Omega)$-regularity.
\medskip

 In order to obtain the
different error estimates, the following regularity hypotheses for
the (unique)  solution
 $({\bf u},p)$ of $(P)$ will be appearing:
\medskip

\noindent $ {\bf (H1)} \quad \quad {\bf u} \in L^{\infty} ({\bf H}^2\cap {\bf V}),
\quad p_t \in  L^2 (H^1), \quad {\bf u}_t \in L^{2}({\bf L}^2) ,
 \quad   {\bf u}_{tt} \in L^{2} ({\bf H}^{-1})
$
\medskip

\noindent $ {\bf (H2)} \quad \quad
p_{tt} \in  L^2 ( H^1), \q
 {\bf u}_{t} \in L^{\infty} ({\bf L}^{3})\cap L^2({\bf H}^1)  ,\q {\bf u}_{tt} \in
L^{2} ({\bf L}^{2}),   \quad {\bf u}_{ttt} \in L^{2} ({\bf
H}^{-1}) $
\medskip

\noindent $ {\bf (H3)} \quad \quad  {\bf u}_{tt} \in
L^{\infty} ({\bf H}^{-1})$
\begin{remark} 
Unfortunately,  to obtain  hypotheses ${\bf (H1)}$-${\bf (H3)}$ is
necessary to assume that ${\bf u}_t (0) \in {\bf H}^1$, which
implies a non-local compatibility condition for the data ${\bf
u}_0$ and ${\bf f}(0)$. In particular, it is proved in \cite{hr} that
${\bf (H1)}$-${\bf (H3)}$ is satisfied (al least locally in time), if there exists  $p_0
\in H^1(\Omega)$ (the initial pressure)  solution
 of the following overdetermined Neumann problem
$$
\left\{
\begin{array}{l}
\Delta  p_0 = \nabla \cdot \Big ( {\bf f}(0) -( {\bf u}_0 \cdot
\nabla) {\bf u}_0 \Big ) \quad \mbox {in } \Omega, \\
\noalign{\smallskip} \nabla  p_0|_{\partial \Omega} = \Big ( \,
\Delta  {\bf u}_0 + {\bf f}(0) -({\bf u}_0 \cdot \nabla ) {\bf
u}_0 \Big ) |_{\partial \Omega}.
\end{array}
\right.
$$
 which in practice is hard to fulfill (see \cite{hr}).
 
In \cite{prohl},  error estimates for the (non-incremental) Chorin-Temam projection scheme are deduced without requiring this non-local compatibility condition,
arriving at the optimal  order $O(k)$  in $l^{\infty} ({\bf
L}^2)$ for the velocity and  in $l^{\infty}(H^{-1})$ for the
pressure, where a weight at the initial time steps must be
included to deduce the optimal order for the pressure (only possible in a negative norm).

Nevertheless, for the incremental scheme Algorithm~\ref{alg-time} it is not clear how to avoid this compatibility on the data  using adequate weights at the initial time steps. 
\end{remark}
\subsection{$O(k)$-error estimates for both velocities}

\begin{theorem} \label{v1}
 Under conditions of Lemma~\ref{vn},
  {\bf (H1)} and the bound for the initial error pressure  $ |\nabla e_{p}^0 | \leq C$,
 the following error estimates hold:
\begin{equation}\label{est-errores-velocidad}
\| \etilde  \| _{ l ^{\infty}( {\bf L}^2) \cap
l ^{2}( {\bf H}^1)}+ \| {\bf e}^{m+1}  \| _{ l ^{\infty}( {\bf
L}^2) \cap l ^{2}( {\bf H}^1)}  \leq C \, k,  \quad  \| e_{p}^{m+1}
\|_{l^{\infty}(H^1)} \le C ,
\end{equation}
\begin{equation}\label{est-diferencia-errores}
 \| \etilde -{\bf e}^m \| _{l^2 ( {\bf L}^2)}+\| \ee-\etilde \| _{l^2 ( {\bf L}^2)}
 \leq C\, k^{3/2}.
\end{equation}
\end{theorem}

\begin{proof}
The proof follows similar lines of \cite{gq} and \cite{shen}.

\noindent 
By multiplying $(E_1)^ {m+1}$ by $ 2 \, k\,\etilde$ and
integrating in $\Omega$, one has:
\begin{equation}\label{23}
\begin{array}{l}
|\etilde|^2 -|{\bf e}^{m}|^2 + |\etilde-{\bf e}^{m}|^2 + 2\, k
\|\etilde\|^2 +2\, k\, \Big ( \nabla  e_{p}^m   , \etilde \Big )
\\
\noalign{\smallskip} 
\quad =  2 \, k \, \Big( \mathcal{E}^{m+1} + {\bf NL}^{m+1}, \etilde
\Big) -2\, k^2\, \Big (
\nabla  \delta_t  p
(t_{m+1}))   , \etilde \Big ).
\end{array}
\end{equation}
On the other hand, multiplying $(E_2)^{m+1}$ by $k ({\bf
e}^{m+1}+\etilde) + k^2\, \Big ( \nabla e_{p}^{m+1} + \nabla
e_{p}^m \Big )$ and using that $\Big ( \ee, \nabla e_p^{m+1} \Big )
=0=\Big ( \ee,
\nabla e_p^m \Big ) =\Big ( \ee, \nabla \delta_t p(t_{m+1}) \Big ) $
(see 
(\ref{L2-orthog})), we obtain
\begin{equation}\label{26}
\begin{array}{c}
\displaystyle |{\bf e}^{m+1}|^2 - |\etilde |^2 +  
|k\, \nabla e_{p}^{m+1}|^2 - |k\, \nabla e_{p}^m |^2 
- 2\, k \,
\Big( \etilde , \nabla e_{p}^m  \Big )
\\
\quad = k^2\, \Big ( \etilde , \nabla\delta_t p(t_{m+1})\Big ) +k^3 \Big ( \nabla \delta_t p(t_{m+1}), \nabla e_p ^{m+1} +
\nabla e_p ^m    \Big ).
\end{array}
\end{equation}
By adding   $ (\ref{23})$ and $(\ref{26})$,
the term $2\, k \Big ( \etilde , \nabla e_p^m \Big )$ vanish,
 obtaining
\begin{equation}\label{24}
\begin{array}{l}
|\ee|^2 -|{\bf e}^{m}|^2 +
|k\, \nabla e_{p}^{m+1}|^2 - |k\, \nabla e_{p}^m |^2 
  + |\etilde-{\bf e}^{m}|^2 + 2\, k
\|\etilde\|^2 
\\
\noalign{\smallskip}
\quad  \le  2 \, k \, \Big( \mathcal{E}^{m+1} + {\bf NL}^{m+1}, \etilde \Big) - k^2\, \Big (
\nabla  \delta_t  p(t_{m+1}))   , \etilde \Big )
\\
\noalign{\smallskip}
\quad +k^3 \Big ( \nabla \delta_t p(t_{m+1}), \nabla e_p ^{m+1} +
\nabla e_p ^m    \Big ).
\end{array}
\end{equation}
The consistency error can be bounded as
follows:
$$
  2 \, k \, \Big( \mathcal{E}_1^{m+1}, \etilde \Big) \le 
\displaystyle \frac{k}{3} \, \| \etilde \| ^2
 + C \,k^2 \int_{t_m}^{t_{m+1}}  \|  {\bf u}_{tt}
\| _ {{\bf H}^{-1}}^2 \, dt ,
$$
$$ 2\, k \, \Big ( \mathcal{E}^{m+1}_2 , \etilde \Big )
\le 2\, k\, \Big ( \int_{t_m}^{t_{m+1}} | {\bf u}_t| \Big ) \, \| \nabla {\bf
u} (t_{m+1}) \|_{L^3} \| \etilde \|_{L^6}
 \le \frac{k}{3} \, \| \etilde \|^2 +
C\,k^2 \, \int_{t_m
  }^{t_{m+1}} | {\bf u}_t |^2 .
$$

 By using the antisymmetry property $c(\widetilde{\bf u}^m , \etilde
 ,\etilde )=0$ and equality (\ref{continuous-dependence-errorsa}), we bound  the convective terms as follows:
 $$  2\, k\, \Big( {\bf NL}^{m+1},
\etilde \Big)= 2\, k \,c \Big( \widetilde{\bf e}^m, {\bf u}(t_{m+1}),
\etilde \Big)        \leq \displaystyle\frac{k}{3} \|\etilde \| ^2
+ C\,  k \, \|{\bf u} (t_{m+1})\|_{L^{\infty} \cap W^{1,3}} ^2
 |\widetilde{\bf e}^m| ^2
$$
$$\le \frac{k}{3} \| \etilde \|^2+C\, k\, |{\bf e}^m|^2+C\, k \, \Big
( |k\, \nabla e_p^m |^2 + |k\, \nabla e_p^{m-1}|^2 +k^2\, |\nabla
\delta_t p(t_{m}) |^2 \Big )
$$

Now, by using that $\Big({\bf e}^m , \nabla\delta_t
p(t_{m+1}) \Big)=0$, we bound the third term at RHS of (\ref{24}):
$$
 -k^2\, \Big (\etilde , \nabla\delta_t p(t_{m+1}) \Big ) =- k^2\, \Big (
 \etilde-{\bf e}^m
 , \nabla\delta_t p(t_{m+1})\Big )
\le \frac14 \,  | \etilde -{\bf e}^m |^2 +C\, k^4 \,|\nabla\delta_t p(t_{m+1})  |^2.
$$
 Finally, we bound the last term at RHS of  (\ref{24}):
$$
k^3 \Big ( \nabla \delta_t p(t_{m+1}), \nabla e_p ^{m+1} +
\nabla e_p ^m    \Big )  = k^3\,  \Big ( \nabla \delta_t p(t_{m+1}),
\nabla e_p ^{m+1} -\nabla e_p^m \Big )+ k^3 \,\Big ( \nabla \delta_t p(t_{m+1}),
  2\, \nabla e_p ^m    \Big )=I_1+I_2
$$
By using $(E_2)^{m+1}$, the  $I_1$-term can be rewritten as
$$I_1 =  k^2\,  \Big ( \nabla \delta_t p(t_{m+1}), \etilde -\ee \Big )+k^4
\, |\nabla \delta_t p(t_{m+1}) |^2 
$$
$$= k^2\,  \Big ( \nabla \delta_t
p(t_{m+1}), \etilde -{\bf e}^m \Big )+k^4
\, |\nabla \delta_t p(t_{m+1}) |^2
$$
$$
\le \frac14 \, | \etilde - {\bf e}^m |^2+ C\, k^3
\int_{t_m}^{t_{m+1}} | \nabla  p_t|^2
$$
We bound $I_2$ as:
$$I_2 \le C\, k\, | k\, \nabla e_p^m | ^2+ C\,k^3 \, | \nabla \delta_t
p(t_{m+1}))| ^2\le C\, k\, | k\, \nabla e_p^m | ^2+ C\,k^2
\int_{t_m}^{t_{m+1}} | \nabla  p_t|^2
$$
By applying these bounds in  (\ref{24}), 
$$
\begin{array}{l}
|\ee|^2 -|{\bf e}^{m}|^2 +
|k\, \nabla e_{p}^{m+1}|^2 - |k\, \nabla e_{p}^m |^2 
  +\displaystyle\frac{1}{2} |\etilde-{\bf e}^{m}|^2 +  k
\|\etilde\|^2 
 \le 
 \displaystyle  C \, k\, |   {\bf e}^m |^2 \\
\noalign{\smallskip}
\displaystyle 
\quad +  C\,k^2
 \int_{t_m}^{t_{m+1}} \Big( \| {\bf u}_{tt} \|_ {{\bf H}^{-1}}^2 +
  |{\bf u}_t |^2 +  | \nabla  p_t|^2 \Big) dt +C\, k\, \Big (| k\, \nabla e_p^m|^2+ | k\, \nabla
  e_p^{m-1}|^2 \Big ).
\end{array}
$$
Adding up from $m=1$ to $r$, 
%
and  applying the  discrete Gronwall inequality, we arrive at:
$$  \|{\bf
  e}^{m+1}\|_{l^{\infty}({\bf L}^2)}
+ \|\etilde\|_{l^{2}({\bf H}^1)}  \leq C\,k, \quad \| \etilde - {\bf
  e}^m\|_{l^{2}({\bf L}^2)} \le C\, k^{3/2} \quad \mbox{ and  }\quad  \| e_{p}^{m+1}
\|_{l^{\infty}(H^1)} \le C .
$$
Finally,  by applying Lemma~\ref{regularidad-errores}, estimates $\|\etilde\|_{l^{\infty}({\bf L}^2)}\le C\, k $ and $\|{\bf e}^{m+1}\|_{l^{2}({\bf H}^1)} \leq  C\,k$ hold.
%
%
\end{proof}

Notice that the error estimate $\|\widetilde {\bf e}^m\|_{l^2({\bf H}^1)}\le
C\,k$ implies in particular the uniform estimates
$$\|\widetilde{\bf e}^m\|_{{\bf H}^1}\le C
\q\hbox{and}\q \|\widetilde{\bf u}^m\|_{{\bf H}^1}\le C\q  \forall\, m.$$

\subsection{$O(k)$-error estimates for the  pressure}

First, we are going to obtain error estimates for the discrete time derivative of velocity, and then the optimal order $O(k)$ for the pressure.
\begin{lemma}[Continuous
  dependence of discrete derivatives for the projection step] \label{regularidad-derivadas-errores} 
  
 \noindent It holds
\begin{equation}\label{dependence-deltaerrors-1}
|\delta_t \etilde |^2= | \delta_t \ee |^2+ | k\, \nabla \delta_t
(e_p^{m+1}-e_p^m -k\, \delta_t p(t_{m+1}) )|^2 , 
\end{equation}
$$
 | \delta_t \ee
-\delta_t \etilde | \le | \delta_t \etilde -
\delta_t {\bf e}^m | ,
$$
$$\| \delta_t {\bf e}^{m+1} \| \leq C\,  \| \delta_t \etilde\| .
$$
\end{lemma}

\begin{proof}
The proof is similar to Lemma \ref{regularidad} and Lemma
\ref{regularidad-errores},  using that $\delta_t \ee= P_{{\bf
H}} (\delta_t \etilde)$.
\end{proof}

\begin{theorem}\label{lv12}
Assuming hypotheses of Theorem \ref{v1},  $({\bf H2})$ and the
following constraints on the first-step approximation
$$
| \delta_t {\bf e}^{1} | + | k \,
 \nabla \delta_t e_{p}^1 | \le C\, k 
 \q\hbox{and}\q
\| \delta_t \widetilde{\bf e}^{1} \|
\le C\, \sqrt{k},
$$
 one has
$$\|\delta_t {\bf e}^{m+1} \| _{l^{\infty}({\bf L}^2)\cap l^2 ({\bf H}^1)} + \| \delta _t
\etilde \| _{l^{\infty}({\bf L}^2)\cap l^2 ({\bf H}^1)} \le C \, k
\quad \mbox{and} \quad \| \delta_t e_p^{m+1} \| _{l^{\infty}( H^1)} \le C.
  $$
\end{theorem}


\begin{proof} By making $\delta_t (E_1)^{m+1}$ and
$\delta_t (E_2)^{m+1}$:
$$ \left.  \begin{array}{c}
\displaystyle\frac{\delta_t \etilde -\delta_t {\bf e}^{m}}{k}
-\Delta  \delta_t \etilde +  \nabla (\delta_t  e_{p}^m +k\,
\delta_t \delta_t p (t_{m+1}))=
 \delta_t \mathcal{E}^{m+1} +\delta_t {\bf NL}^{m+1}
\end{array}
\right.\leqno (D_1)^{m+1}$$
 where $\delta_t \delta_t p(t_{m+1}) = \displaystyle\frac{1}{k}
(\delta_t p(t_{m+1}) -\delta_t p(t_{m}))$, and
$$ \left.  \begin{array}{c}
\displaystyle\frac{\delta_t {\bf e}^{m+1} -\delta_t \etilde}{k} +
\nabla  ( \delta_t e_{p}^{m+1}- \delta_t  e_{p}^m  -k\, \delta_t
\delta_t p(t_{m+1})  ) =0.
\end{array}
\right. \leqno (D_2)^{m+1} $$ The proof follows similar lines of
Theorem \ref{v1}. Multiplying $(D_1)^{m+1}$ by $2\,k\,
\delta_t \etilde $, we get:
\begin{equation}\label{est-derivatives}
\begin{array}{lcr}
 |  \delta_t  \etilde | ^2 -|  \delta_t  {\bf e}^m | ^2 +
 |  \delta _t  \etilde -\delta_t  {\bf e}^m | ^2 + 2\,
 k \|  \delta_t \etilde \| ^2
 +2\, k \Big ( \nabla  \delta_ t e_{p}^m ,  \delta_t \etilde \Big )
 &&\\ 
 \quad =2\, k\,  \Big (\delta_t \mathcal{E}^{m+1} + \delta_t  {\bf NL}^{m+1},\delta_t \etilde
\Big )-2\, k^2\, \Big (\nabla\delta_t\delta_t p (t_{m+1})) , \delta_t \etilde \Big ).&&
\end{array}
\end{equation}
On the other hand,  multiplying $(D_2)^{m+1}$ by $k\, (\delta_t
{\bf e}^{m+1} + \delta_t \etilde )+ k^2 \, (\nabla \delta_t
e_{p}^{m+1} +\nabla \delta_t e_{p}^{m})$,
\begin{equation}\label{8n}
 \begin{array}{lcr}
&& | \delta_t {\bf e}^{m+1} |^2 - | \delta _t \etilde | ^2  +
  |k\, \nabla \delta_t e_{p}^{m+1} | ^ 2 -
 |k\, \nabla \delta_t e_{p}^{m} | ^ 2 -2\, k \, \Big ( \delta_t \etilde,
 \nabla  \delta_t e_{p}^m  \Big )\\
&& =k^2 \,  \Big (  \nabla  \delta_t \delta_t p (t_{m+1}),\delta_t \etilde
 \Big ) +k^3 \, \, \Big (\nabla \delta_t \delta_t p(t_{m+1}) , \nabla
 \delta_t e_p^{m+1} +\nabla \delta_t e_p^m  \Big )  .
\end{array}
\end{equation}
By adding (\ref{est-derivatives}) and (\ref{8n}), the term $2\, k \, \Big ( \delta_t \etilde, \nabla  \delta_t e_{p}^m  \Big )$ cancels, arriving at
\begin{equation}\label{suma-est-derivatives}
\begin{array}{lcr}
 |  \delta_t  \ee | ^2 -|  \delta_t  {\bf e}^m | ^2 +
 |  \delta _t  \etilde -\delta_t  {\bf e}^m | ^2  +
  |k\, \nabla \delta_t e_{p}^{m+1} | ^ 2 -
 |k\, \nabla \delta_t e_{p}^{m} | ^ 2 + 2\,
 k \|  \delta_t \etilde \| ^2
 &&\\ 
 \quad = 2\, k\,  \Big (\delta_t \mathcal{E}^{m+1} + \delta_t  {\bf NL}^{m+1}  ,\delta_t \etilde
\Big )&& \\
\quad - k^2\, \Big (\nabla\delta_t\delta_t p (t_{m+1})) , \delta_t \etilde \Big )+k^3 \, \, \Big (\nabla  \delta_t e_p^{m+1} +\nabla \delta_t e_p^m, \nabla \delta_t \delta_t p(t_{m+1})   \Big )  .&&
\end{array}
\end{equation}
We bound the RHS of (\ref{suma-est-derivatives}) as follows:
$$
2\,k \Big( \delta_t\mathcal{E}^{m+1}_1  , \delta_t  \etilde
\Big) \le
  \varepsilon \,k  \| \delta_t \etilde \| ^2
+ C\, k^2 \int_{t_{m-1}}^{t_{m+1}} \|
 {\bf u}_{ttt} \|_{H^{-1}}^2
$$
\begin{eqnarray*}
2\,k \Big( \delta_t \mathcal{E}^{m+1}_2  , \delta_t  \etilde
\Big) &=& 2\, k \Big( \delta _t  {\bf u}(t_{m+1}) \cdot \nabla
({\bf u}(t_{m+1}) -{\bf u} (t_m)) , \delta _t  \etilde \Big)
 \\
  &+& 2\, k \Big( ( \delta _t  {\bf u}(t_{m+1})
 -  \delta_t  {\bf u}(t_m) )
\cdot \nabla {\bf u} (t_m) , \delta_t \etilde \Big):=I_1+I_2
\end{eqnarray*}
$$I_1  \le \varepsilon  k \| \delta_t \etilde \|^2 +
C\,k \| \delta _t {\bf u}(t_{m+1}) \|_{L^{3}}^2
  \|\int_{t_m}^{t_{m+1}} \partial_t {\bf u} \|_{H^1}^2
\le \varepsilon  k \| \delta_t \etilde \|^2 + C\, k^2 \| {\bf u}_t
\| _{L^{\infty}(L^3)} \int_{t_m}^{t_{m+1}} \| {\bf u}_t \|^2
_{H^{1}}
$$
$$I_2\le \varepsilon \, k \| \delta_t \etilde \|^2 +
C\,k  \, \| \nabla{\bf u}(t_{m})\|_{L^3}^2 \,
 \Big | \delta_t \Big ( \displaystyle \int_{t_m}^{t_{m+1}} {\bf u}_t  \Big ) \Big |^2
 \le
\varepsilon \, k \| \delta_t \etilde \|^2 +C\, k^2 \,
\int_{t_{m-1}}^{t_{m+1}}  | {\bf u}_{tt} |^2
$$
(in the above inequality we have used estimates obtained in \cite{shen2}).

Now, we bound the non-linear terms:
\begin{eqnarray*}
&&2 \, k\, \Big ( \delta_t {\bf NL}^{m+1}, \delta_t \etilde
\Big )= 2\, k\, c \Big ( \delta _t  \widetilde{\bf e}^m , {\bf
u}(t_{m+1}), \delta _t  \etilde       \Big )
 +2\,k\, c \Big (\delta_t \widetilde{\bf u}^m ,
\etilde, \delta _t  \etilde  \Big )
 \\
 &&\quad +2\,k\, c \Big (
\widetilde{\bf e}^{m-1} , \delta_t {\bf u}(t_{m+1}) , \delta _t
\etilde    \Big )
 + 2\,k\,c \Big ( \widetilde{\bf u} ^{m-1} ,   \delta _t \etilde,
   \delta_t \etilde \Big ):=\sum_{i=1}^4 L_i
\end{eqnarray*}
$$L_1\le k\, |   \delta_t \widetilde{\bf e}^m | \, \|{\bf u}
(t_{m+1})\| _{L^{\infty}\cap W^{1,3}}\, \|  \delta_t \etilde\|
\le \varepsilon \, k\, \|  \delta_t \etilde\| ^2 + C \, k  \, |
\delta_t \widetilde{\bf e}^m   | ^2  $$
$$\le  \varepsilon \, k\, \|  \delta_t \etilde\| ^2 + C \, k  \, |
\delta_t {\bf e}^m   | ^2 +C\,k \, \Big (  | k\, \nabla \delta_t e_p^m|^2+ |k\, \nabla \delta_t e_p^{m-1}| ^2 +C\,k^2 \, | \nabla \delta_t \delta_t
p(t_{m})|^2 \Big )
$$
(here  $(\ref{dependence-deltaerrors-1})$ is used),
$$L_2=2\, k\, c \Big (  \delta_t \widetilde {\bf e}^m ,
\etilde, \delta_t \etilde\Big )+ 2\,k\,  c \Big ( \delta_t
{\bf u}(t_{m}), \etilde, \delta_t \etilde\Big
):=L_{21}+L_{22} $$
$$ L_{21} \le 2\, k\, \|\etilde \| \,\| \delta_t \widetilde {\bf e}^m \|_{L^3}  \|
\delta_t \etilde  \| 
\le \varepsilon \, k \Big( \|  \delta_t \etilde \| ^2  + 
\| \delta_t \widetilde {\bf e}^m \|^2 \Big)+ C\, k \,|  \delta_t \widetilde {\bf e}^m |^2
$$
where we have used that $\|\etilde \|\le C$,
$$ L_{22}\le k\,
 \| \delta_t {\bf u}(t_{m}) \| _{L^{3}}\,\| \delta_t \etilde \| \, \|\etilde \|
 \leq \varepsilon \, k\, \|  \delta_t \etilde \| ^2 + C \, k\,
\|\etilde \|^2 $$ (in the above estimate we have used the
regularity ${\bf u}_t \in L^\infty({\bf L}^{3})$), and from a
similar way,
$$ L_3\leq  \varepsilon \, k\, \| \delta_t \etilde \| ^2 + C \,
k\, \|\widetilde{\bf e} ^{m-1}\|^2 .$$

\vspace{.3cm}
\noindent
Finally,
$$L_4 =0.$$
 Reasoning as in Theorem~\ref{v1}, taking into account the above estimates and choising  $\varepsilon$
small enough, we arrive at
$$
\begin{array}{lcr}
|  \delta_t  \ee | ^2 -|  \delta_t {\bf e}^m | ^2 +
\displaystyle\frac{1}{2} |  \delta _t  \etilde -\delta_t  {\bf e}^m | ^2 +
  |k\, \nabla \delta_t e_{p}^{m+1} | ^ 2 -
 |k\, \nabla \delta_t e_{p}^{m} | ^ 2 + 
 k \|  \delta_t \etilde \| ^2
&&\\
\noalign{\smallskip} \dis\le   C\, k |\delta_t {\bf e}^m
|^2  +C\, k^2  \int_{t_{m-1}}^{t_{m+1}} \Big(\| {\bf u}_{ttt} \|
_{H^{-1}}^2 + | {\bf u}_{tt} |^2 \Big )+ C\, k^2 \,   \displaystyle  \int_{t_{m}}^{t_{m+1}}\| {\bf u}_t \|^2 + \frac{k}2 \| \delta_t \widetilde {\bf e}^m \|^2
&&\\
\noalign{\smallskip}
+ C\,
k\Big( \|\widetilde{\bf e} ^{m-1}\|^2
+ \|\widetilde{\bf e}^{m+1}\|^2 \Big)
 +C\,k \, \Big (  | k\, \nabla \delta_t e_p^m|^2+ |k\, \nabla \delta_t
e_p^{m-1}| ^2 )\Big )+C\,k^2 \,\displaystyle\int_{t_{m-1}}^{t_{m+1}} | \nabla
p_{tt} |^2 .&&
\end{array}
$$
 Now, by adding from $m=1$ to $r$ and
  using error estimates of Theorem~\ref{v1},
  we arrive at
 \begin{eqnarray*}
&& | \delta_t  {\bf e}^{r+1} |^2 + |k\,\nabla \delta_t  e_{p}^{r+1} |^2
 + \frac{1}{2}\sum_{m=1}^r  | \delta_t
 \etilde - \delta_t  {\bf e}^m |^2 
 +  \frac{k}{2} \sum_{m=1}^r \| \delta_t \etilde \| ^2  
 \\ &&  \le 
 | \delta_t {\bf e}^1 |^2 + |k\,\nabla \delta_t e_{p} ^1  |^2 +
 \frac{k}2 \| \delta_t \widetilde {\bf e}^1 \|^2
+ C\, k \sum_{m=1}^r \Big ( |
\delta_t {\bf e}^m |^2 + |k\, \nabla
\delta_t e_p^m| ^2 + |k\, \nabla
\delta_t e_p^{m-1}| ^2 \Big )
 +C\, k^2 .
 \end{eqnarray*}


Then,  applying the discrete Gronwall Lemma, we obtain the estimates
$$
\| \delta_t {\bf e}^{m+1}\|_{l^{\infty}(L^2)} \le C\, k , \quad
\sum_{m=1}^r  | \delta_t  \etilde - \delta_t  {\bf e}^m |^2\le C\,
k^2,\q \| \delta_t \widetilde{\bf e}^{m+1}\|_{l^{2}(H^1)} \le C\,
k $$
$$\quad \mbox{and}\quad  \|  \delta_t e_p^{m+1} \|
_{l^{\infty}(H^1)}\le C .$$
 After that,  taking into account
Lemma \ref{regularidad-derivadas-errores},
$$ \|\delta_t \ee \| _{l^{2}(H^1)} \le C\, k  \quad 
\sum | \delta_t \etilde - \delta_ t \ee | ^2 \le C\, k^2
 \quad \mbox{and}\quad
 \| \delta_t \widetilde{\bf e}^{m+1}\|_{l^{\infty}(L^2)} \le C\, k
   $$
hence the proof is finished.
\end{proof}









\begin{theorem}\label{cor1} Under hypothesis of Theorem \ref{lv12} and
  $({\bf H3})$,
the following error estimates hold
\begin{equation}\label{linftyH1-velocity}
\|\etilde \| _{l^{\infty} (H^1)}\, + \,
\|e_p^{m+1}\|_{l^{\infty} (L ^2 )} \le C\, k.
\end{equation}

\end{theorem}

\begin{proof} 

\noindent{\bf Step 1.} To prove 
\begin{equation}\label{l2l2-pressure}
\|e_p^{m+1}\|_{l^{2} (L ^2 )} \le C\, k.
\end{equation}

We are going to deduce the estimate (\ref{l2l2-pressure}) from Theorem \ref{lv12}  and the continuous
inf-sup condition applied to  $ (E_3 ) ^{m+1}$. Indeed, rewritten
$(E_3)^{m+1}$ as
$$-\nabla e_{p} ^{m+1} =\delta_t  {\bf e}^{m+1} - \Delta
\etilde  -\mathcal{E}^{m+1} -{\bf NL}^{m+1}, \quad \ e_p ^{m+1} \in
L_0^2 (\Omega), $$ then, applying the continuous inf-sup condition
\begin{equation}\label{est-l2-error presion}
\begin{array}{l}
\|e_{p} ^{m+1} \|  _{L^2}
 \le  C  \Big \{ \| \delta_t  {\bf e} ^{m+1}  \|
_{H^{-1}} + \| \etilde \|+
 \|\mathcal{E}^{m+1} \|_{H^{-1}}  +  \|{\bf NL}^{m+1} \|_{H^{-1}}
 \Big\}
\\
 \qquad \le  C  \Big \{ \| \delta_t  {\bf e} ^{m+1}  \|
_{H^{-1}} + \| \etilde \|+  \| \widetilde{\bf e}^m  \|
+  k \| {\bf u}_{tt} \|_{L^\infty(0,T;H^{-1})} 
+  k \| {\bf u}_{t} \|_{L^\infty(0,T;L^{3})} \Big \},
\end{array}
\end{equation}
where we have used the estimate
$$\| {\bf NL}^{m+1}\|_{H^{-1}} 
\le C ( \| \widetilde{\bf e}^m \|\, \|{\bf u}(t_{m+1})\|_{L^3} + \| \etilde \| \,\|\widetilde {\bf u}^{m}\|_{L^3} )
 \le C ( \| \widetilde{\bf e}^m \|  + \| \etilde \| ).$$
By taking into account that  $\| \etilde \|_{l^{2}(H^1)} \le C\,k$ and $ \| \delta_t \ee \|
_{l^{\infty} (L^2)} \le C\, k$ and hypothesis {\bf (H2)} and {\bf (H3)}, we arrive at \eqref{l2l2-pressure}.
\medskip

\noindent{\bf Step 2.} 
To prove (\ref{linftyH1-velocity}) for $\etilde$.  

From $(E_3)^{m+1}$ we have
$$- \Delta \etilde = -\delta_t {\bf e}^{m+1} -\nabla  e_{p} ^{m+1} +
\mathcal{E}^{m+1} + {\bf NL}^{m+1}, \qquad \etilde
  |_{\partial \Omega } =0.$$

  Multiplying  by  $2\,k\,  \delta _t \etilde $,
  we obtain
\begin{equation}\label{step-2}
\begin{array}{l}
  |\nabla \etilde |^2 - | \nabla \widetilde{\bf  e}^{m}|^2 +  |\nabla \etilde - \nabla \widetilde{\bf  e}^{m}|^2
\\
=
2\, k \Big(-\nabla e_{p} ^{m+1} -  \delta_t {\bf e}^{m+1} + \mathcal{E}^{m+1} + {\bf NL}^{m+1},  \delta _t \etilde       \Big)
\\
 \le C\,k \, |e_{p} ^{m+1} |^2
 +C\,k \,  |\nabla \delta_t \etilde|^2 +C\,k\, |\delta_t {\bf
e} ^{m+1} |^2 +
 C\, k\, \|\mathcal{E}^{m+1} \|_{H^{-1}}^2  +C\, k\,  \|{\bf NL}^{m+1} \|_{H^{-1}}^2
 \\
  \le C\,k \, |e_{p} ^{m+1} |^2
 +C\,k \,  |\nabla \delta_t \etilde|^2 +C\,k\, |\delta_t {\bf
e} ^{m+1} |^2 +
 C\, k^3  +C\, k\, ( \| \widetilde{\bf e}^m \| ^2 + \| \etilde \|^2 )
 \end{array}
\end{equation}
where we have bounded the two last terms at RHS of (\ref{step-2}) as in (\ref{est-l2-error presion}). Adding (\ref{step-2}) from $m=0$ to $r$  and  applying the estimates of Theorems \ref{v1} and \ref{lv12} and (\ref{l2l2-pressure}), we arrive at (\ref{linftyH1-velocity}) for $\etilde$.
\medskip

\noindent{\bf Step 3.} To prove  (\ref{linftyH1-velocity}) for $e_p^{m+1}$.

By using  the inequality
(\ref{est-l2-error presion}) and taking into account that
$$ \|
\delta_t {\bf e} ^{m+1}  \| _{H^{-1}} \le C\,  | \delta _t {\bf e}
^{m+1}| \le C\, k \q\hbox{and}\q \| \etilde \| \le C\, k,
$$ we arrive at  (\ref{linftyH1-velocity}) for $e_p^{m+1}$.
 \end{proof}

\subsection{Adittional estimates}




Now, we are going to obtain some $H^2$ stability estimates which will be necessary in next Section  to get  optimal error estimates in space.


%
\begin{lemma}\label{reg-h2}
Under hypotheses of Theorem \ref{v1} and ${\bf (H0)}$, one has
$$ \| \etilde \|_{{\bf H}^2} \le C ,\qquad \forall m .$$
\end{lemma}
\begin{proof} From the $H^2$-regularity of the Poisson problem
$(E_1)^{m+1}$, one has
\begin{equation}\label{unif-estim-veloc}
\|\etilde \|_{{\bf H}^2}^2 \le C\left( \Big |\displaystyle \frac{\etilde -{\bf
    e}^m }{k} \Big | ^2 + | \nabla e_p ^m|^2 + k^2 | \nabla \, \delta_t
p(t_{m+1})| ^2+ | \mathcal{E}^{m+1} | ^2 + |{\bf NL}^{m+1}|^2
\right).
\end{equation}
The first and second term of the RHS of (\ref{unif-estim-veloc}) are bounded  using that $|
\etilde -{\bf e}^m | \le C\, k $ from
(\ref{est-diferencia-errores}) and $ \| e_{p}^{m+1}
\|_{l^{\infty}(H^1)} \le C$ from (\ref{est-errores-velocidad}). It
is easy to bound the third and the forth term of the RHS of (\ref{unif-estim-veloc}).
Finally, we bound the nonlinear term as follows
$$
\begin{array}{rcl}
 |{\bf NL}^{m+1}|^2 
&\le& C \Big( \| \widetilde{\bf e}^m \|^2_{{\bf L}^{\infty} \cap {\bf W}^{1,3}}  \| {\bf u}(t_{m+1}) \|^2
+  \| \widetilde{\bf u}^m \|^2 \| \etilde \|^2_{{\bf L}^{\infty} \cap {\bf W}^{1,3}} \Big)
\\ 
 &\le& C\Big(\| \widetilde{\bf e}^{m} \|\,   \| \widetilde{\bf e}^{m} \|_{{\bf H}^2} + \| \widetilde{\bf e}^{m+1} \|\,   \| \etilde \|_{{\bf H}^2}\Big) \le \varepsilon \Big(\| \widetilde{\bf e}^{m} \|_{{\bf H}^2}^2 + \| \etilde \|_{{\bf H}^2}^2  +C\Big) .
 \end{array}
$$
 Then, by applying these estimates in (\ref{unif-estim-veloc}) and taking a small enough $\varepsilon$,  there exists $\alpha<1$ such that
  $$
\|\etilde \|_{{\bf H}^2}^2 \le \alpha \| \widetilde{\bf e}^{m} \|_{{\bf H}^2}^2 + C,
 $$
 hence, by an induction process,
 $$
  \|\etilde\|_{{\bf H}^2}^2 \le \alpha^{m+1}  \| \widetilde{\bf e}^{0} \|_{{\bf H}^2}^2+ C\, (\alpha^m+\cdots +\alpha + 1) \le C
 $$
  and the proof is concluded.
  \end{proof}

\begin{remark}\label{reg-fuerte} As a consequence of  the $l^\infty$ in time estimates $\|\widetilde{\bf e}^{m+1}\|_{{\bf H}^2}\le C$ and $\| e_p^{m+1} \| \le C, \ \forall m $,
one also   has
$$ \| \utilde \| _{{\bf H}^2} \le C \  \mbox{ and } \ \ \| p^{m+1} \| \le C \quad \forall m.
$$
\end{remark}

On the other hand, as a direct consequence of Theorem \ref{lv12}, one has
$$\| \delta_t \delta_t \ee \|_{l^{\infty}({\bf L}^2)}+
\| \delta_t \delta_t \etilde \|_{l^{\infty}({\bf L}^2)}\le C. $$
In particular, using that ${\bf u}_{tt} \in  L^{2}({\bf L}^2)$ (see  {\bf (H2)}), this estimate can be extended to the scheme as
\begin{equation}\label{reg-delta-delta}
\| \delta_t \delta_t \utilde \|_{l^{2}({\bf L}^2)}\le C .
\end{equation}


\begin{lemma}\label{reg-delta-h2}
Under hypotheses of Theorem \ref{lv12} and ${\bf (H0)}$, one has
 $$\| \delta_t \etilde \|_{l^2({\bf H}^2)} \le C. $$ 
In particular $\|
 \delta_t \utilde \|_{l^2({\bf H}^2)} \le C.$
\end{lemma}



\begin{proof} The idea is to argue as in Lemma \ref{reg-h2}, using  the $H^2$-regularity of the Poisson problem
$(D_1)^{m+1}$ and applying Theorem \ref{lv12}.
\end{proof}

\section{ Fully discrete scheme (Algorithm~\ref{algo-fully})}
In this section,  we will denote by $C$ different constants, always
independent of $k$ and $h$.
\subsection{Finite element approximation and fully discrete scheme}

We  consider a segregated FE approximation of the time discrete
Algorithm~\ref{alg-time}. We restrict ourselves to the case where
$\Omega$ is a $2D$ polygon or a $3D$ polyhedron satisfying the
regularity hypothesis  $({\bf H0})$. We consider two FE spaces
 ${\bf Y}_h \subset {\bf H}_0^1
(\Omega)$ and $Q_h \subset H^1(\Omega)\cap L_0^2 (\Omega)$
associated to a regular family of triangulations ${\cal T}_h$ of
the domain $\Omega$ of mesh size $h$ (regular in the Ciarlet's sense 
\cite{ciarlet}). For simplicity, we restrict ${\bf Y}_h$ and $Q_h$ to
globally continuous functions and locally polynomials of degree at
least $1$. Finally, we will assume:

\begin{enumerate}
\item The inverse inequality $ \| {\bf u}_h \| \le C\,h^{-1}  |{\bf u}_h   | $ for each ${\bf u}_h \in {\bf Y}_h$ holds.
\item The stable ``{\it inf-sup}''
condition (\cite{gr}) for $({\bf Y}_h,Q_h)$:
There exists $ \beta >0 $ independent of $ h$ such that,
\begin{equation} \label{cinfsup}
\inf_{q_h \in Q_h \setminus\{0\}} \left( \sup_{{\bf v}_h\in {\bf Y}_h \setminus\{0\}}
\frac{(q_h, \nabla \cdot {\bf v}_h)}{\|{\bf v}_h \| \, |q_h
|}\right) \ge \beta .
\end{equation}
\item There exists  some  interpolation operators with the following properties:
\begin{enumerate}
    \item $I_h : {\bf L}^2 \rightarrow {\bf Y}_h  $ such as
\begin{equation}\label{orthog-interp-u}
    ({\bf u}-I_h {\bf u}, \nabla q_h)=0 , \quad \forall \, q_h \in Q_h
\end{equation}
 satisfying the  approximation properties:
 \begin{equation} \label{H-1approx}
 \|  {\bf u}- I_h   {\bf u} \|_{H^{-1}}  \le C\, h\,
\| {\bf u}\|_{L^2} \quad \forall \,{\bf u} \in
{\bf L}^2 (\Omega),
\end{equation}
  $$  \|  {\bf u}- I_h   {\bf u} \|_{L^2}  \le C\, h\,
\| {\bf u}\|_{H^1} \quad \forall \,{\bf u} \in
{\bf H}_0^1 (\Omega),
$$
 $$  \|  {\bf u}- I_h   {\bf u} \|_{H^1}  \le C\, h\,
\| {\bf u}\|_{H^2} \quad \forall \,{\bf u} \in
{\bf H}^2 (\Omega)\cap {\bf H}_0^1(\Omega),
$$
and the  stability property:
$$  \|  I_h   {\bf u} \|_{H^1}  \le C\, 
\| {\bf u}\|_{H^1} \quad \forall \,{\bf u} \in
{\bf H}_0^1 (\Omega).
$$
    \item $J_h : H^1(\Omega) \cap L_0^2 (\Omega)  \rightarrow Q_h  $ defined by
$$\Big(\nabla (J_h p-p), \nabla q_h \Big)=0 \quad \forall\, q_h \in Q_h ,$$
satisfied the approximation property
 $$ \|p- J_h p \|_{L^2} \le C\, h \, \| p\|_{H^1}
  \quad \forall\, p \in H^1(\Omega) \cap L_0^2 (\Omega).
$$
\end{enumerate}
\end{enumerate}
\begin{remark}[Choice of  $I_h$]
For instance, if we consider the  $\P_1\hbox{-bubble}\times \P _1$
approximation to construct the space ${\bf Y}_h\times Q_h$, then a
possible manner to choose  $I_h$ is as follows: Let $\widetilde I_h$
be a regularization interpolation operator (of Cl\'ement or
Scott-Zhang type) onto the globally continuous and locally $\P _1$
FE space, that is $\widetilde{I}_h {\bf u} \in C^0 (\overline{\Omega})$ and $\widetilde{I}_h {\bf u}|_{T} \in \P _1$ for each
$T\in {\cal T}_h$. Then, $\widetilde I_h$ satisfies
\begin{equation}\label{interp-properties}
 |\widetilde I_h {\bf u}-{\bf
u}|\le C\,h\, \| {\bf u} \|,\quad \|\widetilde I_h {\bf u}-{\bf
u}\|\le C\,h\, \| {\bf u} \|_{H^2}.
\end{equation}
 We define
 $I_h {\bf u} = \widetilde I_h {\bf u} + R_h {\bf u}$, where $R_h {\bf u}= \sum _{T}
\b_T \, \alpha_T({\bf u})$ with $\b_T$ a bubble function and
$\alpha_T\in \R^3$ such as 
\begin{equation}\label{zero-local-average}
\int_T ({\bf u} -I_h {\bf u})=0,\quad  \forall\,T\in {\cal T}_h,
\end{equation}
that is
$$
\alpha _T ({\bf u}) = \frac{\int _T ({\bf u} -\widetilde I_h {\bf
u} )}{\int_T \b _T}\quad \forall\, T\in {\cal T}_h.
$$
 Then,  (\ref{orthog-interp-u}) can be deduced from (\ref{zero-local-average}). 
 Moreover, by using again (\ref{zero-local-average}), 
   it is  known by means of a duality argument (\cite{gr}) that
    $$\| {\bf u}-I_h {\bf u} \|_{H^{-1}} \le C\,
 h\, | {\bf u}-I_h {\bf u} |.$$

 Now, in order  to obtain estimate (\ref{interp-properties}) but 
 changing $ \widetilde I_h$ by $I_h$ it suffices to prove 
$$| R_h ({\bf u}) | \le C\, h\, \| {\bf u} \|\quad\hbox{and}  \quad | \nabla \, R_h
{\bf u} | \le C\, h \, \| {\bf u} \|_{{\bf H}^2}.$$
Indeed, by using orthogonality of the bubble functions,
\begin{eqnarray*}
 \| R_h ({\bf u}) \|_{L^2} ^2 &=& \sum_{T}
 \Big | \alpha_T ({\bf u}) \Big | ^2 \, \| \b_T \|_{L^2}^2 =
 \sum_T \Big (  \int _T {\bf u} -\widetilde I_h {\bf u}  \Big
)^2 \frac{\int_T |\b _T|^2 }{\Big ( \int_T \b_T \Big )^2}
\\
&\le& C \sum_T  |T| \Big ( \int_T | {\bf u}- \widetilde  I_h {\bf u} | ^2 \Big ) \,
 \displaystyle\frac{ |T|}{| T| ^2} \le C \, \| {\bf u} - \widetilde
 I_h {\bf u}\|_{L^2}^2
\end{eqnarray*}
hence $| R_h ({\bf u}) | \le C\, h\, \| {\bf u} \|$, owing to the
approximation property $ |{\bf u}-\widetilde I_h {\bf u} | \le C\,
h\, \| {\bf u}\|$.

Taking the $L^2$-norm of the gradient,
$$ \| \nabla R_h ({\bf u}) \|_{L^2} ^2 = \sum_T \Big (  \int _T {\bf u} -\widetilde I_h {\bf u}  \Big)^2 \frac{\int_T |\nabla \b _T|^2 }{\Big ( \int_T \b_T \Big )^2}
 \le C \sum_T  |T| \Big ( \int_T | {\bf u}- \widetilde  I_h {\bf u} | ^2 \Big ) \,
 \displaystyle\frac{ 1}{| T| ^2} \le \frac{C}{h^2} \, \| {\bf u} - \widetilde
 I_h {\bf u}\|_{L^2}^2 $$
 hence $ \| \nabla R_h ({\bf u}) \|_{L^2} ^2 \le C\, h^2 \| {\bf u} \|_{H^2}^2$, owing to the approximation property $\| {\bf u} - \widetilde I_h {\bf u}\|_{L^2} \le C\, h^2 \| {\bf u} \|_{H^2}.$
\end{remark}
\bigskip 

\noindent Now, following the equality $ \uu= \utilde -k \, \nabla
(p^{m+1} -p^m )$, we define:
\begin{equation}\label{def-interp-aux}
K_{h,k} \uu := I_h
\utilde - k \, \nabla J_h (p^{m+1} -p^m).
\end{equation}
Note that $K_{h,k} \uu \in {\bf Y}_h + \nabla
Q_h$. By  comparing (\ref{def-interp-aux}) with the time discrete Algorithm~\ref{alg-time}:
$$\uu -K_{h,k} \uu = \utilde -I_h \utilde -k \, \nabla \Big ( (p^{m+1}
-J_h p^{m+1}) - (p^m- J_h p^{m})  \Big ),$$
hence, using the $L^2$ approximation property for $I_h$, the $H^1$-stability for $J_h$ and the $H^2 \times H^1$ estimates for $(\widetilde{\bf u}^{m+1}, p^{m+1})$:
$$
|\uu -K_{h,k} \, \uu |  \le C\,\Big( h^2 \,\| \utilde \|_{H^2} +k\,
\|p^{m+1} -p^m \| \Big) \le C (k+h^2)  \quad \forall m.
$$


The fully discrete scheme is described in Algorithm~\ref{algo-fully}.
\begin{algorithm}[htbp]
\begin{description}
\item[Initialization:]  Let $(\widetilde{\bf u}_h^0,p_h^0) \in {\bf Y}_h \times Q_h$ be an approximation of $({\bf u}^0, p^0)$. Put ${\bf u}_h^0=\widetilde{\bf u}_h^0$.
\item[Step of time $m+1$:]  Let $(\widetilde{\bf u}_h^m ,p_h^m)\in {\bf
Y}_h \times Q_h$ and ${\bf u}_h^m \in {\bf Y}_h+\nabla Q_h $ be given.
\begin{description}
\item[\bf Sub-step 1]: 
 Find  $\widetilde{\bf u}_h^{m+1}\in {\bf Y}_h$  such  that,
$$
\Big(\frac{\widetilde{\bf
u}_h^{m+1}-{\bf u}_h^{m}}{k},  {\bf v}_h\Big) + c\Big( \widetilde{\bf
u}_h^{m}, \widetilde{\bf u}_h^{m+1},  {\bf v}_h\Big)
 +  \Big(\nabla \,  \widetilde{\bf u}_h^{m+1},  \nabla \, {\bf v}_h\Big)
 +\Big(\nabla p_h ^{m},   {\bf v}_h\Big)  =
\Big({\bf f}^{m+1},  {\bf v}_h\Big).
\leqno{(S_1)^{m+1}_h}
$$
\item[\bf Sub-step 2]: Find $p_h^{m+1}\in Q_h$ such that
$$\Big(k\, \nabla (p_h^{m+1}- p_h^{m} ) , \nabla q_h\Big )= \Big(\widetilde{\bf u}_h ^{m+1},
\nabla q_h \Big) \quad \forall q_h \in Q_h .
\leqno{(S_2)^{m+1}_{a,h}}$$
Now, we  define ${\bf u}_h^{m+1}\in {\bf Y}_h+\nabla Q_h  $ by
$$
\uu_h = \utilde_h -k\, \nabla (p_h^{m+1} -p_h^m).
\leqno{(S_2)^{m+1}_{b,h}}$$
\end{description}
\end{description}
\caption{Fully discrete algorithm}\label{algo-fully}
\end{algorithm}

Notice that, adding both sub-steps of Algorithm~\ref{algo-fully}, we obtain:
$$
\Big(\frac{{\bf u}_h^{m+1}-{\bf u}_h^{m}}{k},  {\bf v}_h\Big) +
c\Big( \widetilde{\bf u}_h^{m}, \widetilde{\bf u}_h^{m+1},  {\bf
v}_h\Big)
 +  \Big(\nabla \widetilde{\bf u}_h^{m+1},  \nabla {\bf v}_h\Big)
 +\Big(\nabla p_h ^{m+1},{\bf v}_h\Big)
 =
\Big({\bf f}^{m+1},  {\bf v}_h\Big).
\leqno{(S_3)^{m+1}_h}
$$
From $(S_2)^{m+1}_{b,h}$, one has the orthogonality property
\begin{equation}  \label{L2-orthog-discrete}
 \Big(\uu_h , \nabla q_h\Big)=0 \quad \forall q_h \in Q_h .
\end{equation}
\begin{remark}[Segregated version of Algorithm~\ref{algo-fully}]\label{segr-Al-2}
 We introduce the end-of-step velocity ${\bf u}_h^{m}$
only for doing the numerical analysis. For 
practical implementations, this velocity ${\bf u}_h^{m}$ can be eliminated, rewriting Algorithm~\ref{algo-fully} as follows:
\medskip

Let $(p_h^{m-1} , p_h^m,  \widetilde{\bf u}_h^m)\in Q_h\times Q_h\times {\bf
Y}_h$ be given.

\begin{description}
\item (a) Find $\utilde_h \in {\bf Y}_h $ such that,  $ \forall \,{\bf v}_h \in
{\bf Y}_h $:
$$ \Big ( \frac{\utilde_h -\widetilde {\bf u}_h^m}{k}, {\bf v}_h \Big )
+ c\Big( \widetilde{\bf u}_h^{m}, \widetilde{\bf u}_h^{m+1},  {\bf
v}_h\Big)
 +  \Big(\nabla \,  \widetilde{\bf u}_h^{m+1},  \nabla \, {\bf v}_h\Big) +
 \Big(\nabla ( 2 p_h ^{m} -p_h ^{m-1}),   {\bf v}_h\Big)  =
\Big({\bf f}^{m+1},  {\bf v}_h\Big).
$$

\item (b) Find $p_h^{m+1}\in Q_h$ such that, $\forall \,q_h \in Q_h$:
$$\Big(k\, \nabla (p_h^{m+1}- p_h^{m} ) , \nabla q_h\Big )=\Big (\widetilde{\bf u}_h ^{m+1},
\nabla q_h\Big )  .$$
\end{description}
Then,  computations for pressure $p_h^{m+1}$ and velocity $\utilde_h $ are decoupled. In
fact, $(a)$ is a linear convection-diffusion-Dirichlet problem for $\utilde_h $ (where each  component of $\utilde_h$ is 
also decoupled from the other ones)  and $(b)$  is a Poisson-Neumann problem for $p_h^{m+1}$. Therefore, Algorithm~\ref{algo-fully} can be rewritten as a fully decoupled scheme.
 
Note that, in order to initialize the scheme we have to start with  a pressure $p_h^{-1}$ which has not sense. We can avoid it starting    from 
an  auxiliary initial step given by either one-step scheme or  by  the scheme written as Algorithm~\ref{algo-fully}, i.e.,
given  $\widetilde{\bf u}^0_h$, $p^0_h$ and ${\bf u}^0_h= \widetilde{\bf
u}^0_h$, we compute first ${\widetilde{\bf u}}_h^1$ from $(S_1)_h^1$ and after  $p_h^1$ from $(S_2)_{a,h}^1$.
\end{remark}
\subsection{Stability and convergence of Algorithm~\ref{algo-fully}}
It is easy to  extend the results given in the previous Section
about the continuous dependence of the projection step  of Algorithm~\ref{alg-time}  to the fully discrete Algorithm~\ref{algo-fully}. Indeed, from
$(S_2)_{b,h}^{m+1}$ and the orthogonality property (\ref{L2-orthog-discrete}), we have
\begin{equation}\label{est-s2}
 |\utilde_h |^2= |\uu_h |^2 + | k\, \nabla (p_h^{m+1} -p_h^m)|^2
\end{equation}
hence, in particular, $ | \uu_h| \le | \utilde_h |$. From
$(S_2)_{a,h}^{m+1}$
$$ | k\, \nabla (p_h^{m+1} -p_h^m)| ^2 =(\utilde_h, k\,\nabla (p_h^{m+1}
-p_h^m ) ) =(\utilde_h -{\bf u}_h^m , k\,\nabla (p_h^{m+1}-p_h ^m
)), $$ hence
$$|\uu_h -\utilde_h | \le | \utilde_h -{\bf u}_h^m |.
$$

Moreover,
 using the antisymmetric property $c(\widetilde{\bf u}_h^m ,
 \widetilde{\bf u}_h^{m+1},\widetilde{\bf u}_h^{m+1})=0$
 (see (\ref{antisim})), one can extend the 
stability and convergence results of Algorithm~\ref{alg-time}
 to the fully discrete Algorithm~\ref{algo-fully}. In particular, for any $r<N$, the following
 stability estimates
 hold:
\begin{equation}\label{reg-debil-esq}
 \begin{array}{c}
\| {\bf u}_h^{r+1} \| _{l^{\infty}(L^2)} +
 \| \widetilde{\bf u}_h^{r+1} \| _{l^{\infty}(L^2)\cap l^2 (H^1)} + \| k\,
 \nabla p_h^{r+1} \|_{l^{\infty}(L^2)} \le C,
\\
\displaystyle \sum_{m=0}^{r} | \utilde_h -{\bf u}_h^m | ^2 +
\sum_{m=0}^{r} | \uu_h - \utilde_h | ^2 \le C.
\end{array}
\end{equation}
Indeed, by making $\Big ( (S_1)_h^{m+1} ,2\,k\, \umedio_h \Big )$,
using the fact that $$2\, k\, (\nabla p_h^m ,\utilde_h )= 2 (k\,
\nabla p_h^m , k \, \nabla (p_h^{m+1} -p_h^m)),$$ and the
equalities $(a-b)2a=a^2-b^2+(a-b)^2$ and
$(a-b)2b=a^2-b^2-(a-b)^2$, we have
\begin{equation}\label{est-s1}
\begin{array}{l}
| \utilde_h |^2 -|{\bf u}_h^m| ^2 + | \utilde_h -{\bf u}_h^m | ^2+ k\,
\| \utilde_h \| ^2 + | k\, \nabla p_h^{m+1} |^2- | k\, \nabla p_h^m|
^2 
\\
\quad - |k\, \nabla (p_h ^{m+1} -p_h^m )| ^2 \le k\,\|{\bf f}^{m+1}
\|_{H^{-1}}^2
\end{array}
\end{equation}
Adding (\ref{est-s2}) and (\ref{est-s1}), the negative term $-|k\,
\nabla (p_h^{m+1} -p_h^m) | ^2$ of (\ref{est-s1}) cancel and we arrive at
$$ | {\bf u}_h^{m+1} |^2 -|{\bf u}_h^m| ^2 + | \utilde_h -{\bf u}_h^m | ^2
 + | k\, \nabla p_h^{m+1} |^2 - | k\, \nabla
p_h^m| ^2 + k\,\| \utilde_h \|^2 \le k\, \|{\bf
f}^{m+1}\|_{H^{-1}}^2$$ Now, adding from ${m=0}$ to $r $ ($r<N$),
 we obtain the desired stability estimates
(\ref{reg-debil-esq}).
\subsection{Problems related to the spatial errors}
We will present  an error analysis for the fully discrete Algorithm~\ref{algo-fully}
$(\umedio_h , \uu _h , p_h^{m+1})$ as an approximation of the time
discrete  Algorithm~\ref{alg-time} $(\umedio, \uu , p^{m+1})$. Consequently, we
define the following  errors:
$$\ee _d =\uu -\uu_h, \qquad \emedio_d =\umedio -\umedio _h ,\quad e_{p,d}^{m+1} = p^{m+1} - p_h^{m+1} .
$$
Splitting the
discrete part  and the interpolation one:
$$\ee_d  = \ee_h   + \ee_i  , \qquad \emedio_d =
\emedio_h +\emedio_i , \qquad e_{p,d}^{m+1} = e_{p,h}^{m+1} +
e_{p,i}^{m+1} $$
 where ${\bf e}_i$ are interpolation errors and ${\bf e}_h$ space discrete
errors, concretely
$$  \ee_h = K_{h,k} \uu -\uu_h  \mbox{\  and \ } \ee_i = \uu -K_{h,k} \uu ,$$
$$
\emedio_h = I_h \umedio -\umedio_h  \mbox{\  and \ }  \emedio_i =
\umedio -I_h \umedio ,$$
$$  e_{p,h}^{m+1} = J_h p^{m+1} - p_h^{m+1}  \mbox{\  and \ }
e_{p,i}^{m+1} = p^{m+1} -  J_h p^{m+1}.
$$


%

\begin{remark}
From the equalities $ \uu = \utilde -k \, \nabla (p^{m+1} -p^m) $
and
 $K_{h,k} \uu = I_h \utilde -k \, \nabla  J_h (p^{m+1} -p^m )$,
one has
\begin{equation}\label{efinal-i}
 \ee_i = \etilde_i  -k \, \nabla (e_{p,i}^{m+1} -e_{p,i}^m).
\end{equation}
In particular, subtracting $\etilde_i$ and (\ref{efinal-i}) replacing $m$ for $m-1$, we get
\begin{equation}\label{ei-interm}
\frac{1}{k}(\etilde_i -{\bf e}_i^m) = e_i (\delta_t \utilde) +
\nabla(e_{p,i}^m -e_{p,i}^{m-1}),
\end{equation}
where $e_i (\delta_t \utilde)=(\etilde_i-\widetilde{\bf e}_i^m)/k$.
 Moreover,
owing to the choice of the interpolation operators $I_h$ and
$J_h$, from (\ref{efinal-i})
\begin{equation} \label{div-nula-ei}
  \Big (\ee_i , \nabla q_h \Big ) =
  \Big(\etilde_i, \nabla q_h \Big) \\
   - k \, \Big ( \nabla (e_{p,i}^{m+1} -e_{p,i}^m), \nabla q_h \Big )
    =0,\quad \forall\, q_h \in Q_h .
\end{equation}
On the other hand, since  $ \Big(\uu_h , \nabla q_h \Big)=0$ 
$\forall\, q_h \in Q_h $ and    $ \Big(\uu , \nabla q\Big )=0$ 
$\forall\, q\in H^1\cap L_0^2 $, then
\begin{equation}\label{div-nula-ed}
 \Big(\ee_d, \nabla q_h \Big)=0 \quad \forall q_h \in Q_h.
\end{equation}
Finally, from (\ref{div-nula-ei}) and (\ref{div-nula-ed}), we arrive
at
$$ \Big(\ee_h, \nabla q_h \Big)=0 \quad \forall q_h \in Q_h.
$$

\end{remark}


By comparing $(S_1)^{m+1},  (S_2)^{m+1}$ and $(S_1)^{m+1}_h,
(S_2)^{m+1}_{b,h}$, we have the following  problems
satisfied by the spatial errors $\emedio_d$ and
$(\ee_d,e_{p,d}^{m+1})$ respectively:
$$
\frac{1}{k}\Big(\emedio_d- {\bf e}_d^{m} , {\bf v}_h \Big)
 +  \Big( \nabla\,   \emedio_d , \nabla\,  {\bf v}_h \Big)
+ \Big(\nabla e_{p,d}^m ,  {\bf v}_h \Big) = {\bf NL}_h^{m+1}({\bf
v}_h) ,\quad \forall\, {\bf v}_h \in {\bf Y}_h,
$$
and
$$
{\bf e}_d^{m+1}= \emedio_d - k  \nabla (e_{p,d}^{m+1} -e_{p,d}^m
),
$$
 where
$${\bf NL}_h^{m+1}({\bf v}_h)= c\Big(\widetilde{\bf u}_h^m , \umedio _h ,
{\bf v}_h \Big)- c\Big( \widetilde{\bf u}^m , \umedio , {\bf v}_h
\Big)=-c\Big(\widetilde{\bf e}_d^m , \umedio , {\bf v}_h \Big)-
c\Big( \widetilde{\bf u}_h^m , \emedio_d , {\bf v}_h \Big).
$$
By splitting the error in the discrete and the interpolation
parts and  using (\ref{efinal-i}) and (\ref{ei-interm}), 
$$\left \{
\begin{array}{l}
\displaystyle\frac{1}{k}\Big(\emedio_h- {\bf e}_h^{m} , {\bf v}_h
\Big)
 +  \Big( \nabla\,   \emedio_h , \nabla\,  {\bf v}_h \Big)+\Big( \nabla
 e_{p,h}^m  ,  {\bf v}_h \Big)= {\bf NL}_h^{m+1}({\bf v}_h)
\\
\displaystyle  - \Big( e_i (\delta_t \utilde), {\bf v}_h\Big )
 - \Big(\nabla \etilde _i , \nabla {\bf v}_h \Big)- \Big(
 \nabla (2\, e_{p,i}^m -e_{p,i}^{m-1} ), {\bf v}_h\Big)
  ,\quad \forall\, {\bf v}_h \in {\bf Y}_h
\end{array} \right.
\leqno{(E_1)^{m+1}_h}
$$
$$
 \ee_h  = \emedio_h - k  \nabla
(e_{p,h}^{m+1} -e_{p,h}^m )  . \leqno{(E_2)^{m+1}_h}
$$
 Finally,
adding  $(E_1)_h^{m+1}$ and  $(E_2)_h^{m+1}$,
$$\left \{
\begin{array}{l}
\displaystyle\frac{1}{k}\Big(\ee_h- {\bf e}_h^{m} , {\bf v}_h
\Big)
 +  \Big( \nabla\,   \emedio_h , \nabla\,  {\bf v}_h \Big)+\Big( \nabla
 e_{p,h}^{m+1}  ,  {\bf v}_h\Big)= {\bf NL}_h^{m+1}({\bf v}_h)
\\
\displaystyle -\Big( e_i (\delta_t \utilde) , {\bf v}_h \Big)
 - \Big(\nabla \etilde _i , \nabla {\bf v}_h \Big)- \Big(
 \nabla (2\, e_{p,i}^m -e_{p,i}^{m-1} ), {\bf v}_h\Big)
  ,\quad \forall\, {\bf v}_h \in {\bf Y}_h.
\end{array} \right.
\leqno{(E_3)^{m+1}_h}
$$


\subsection{ $O(h)$  error estimates for  $\etilde_h$
in $l ^{\infty}( {\bf L}^2) \cap l ^{2}( {\bf H}^1)$ and for $ {\bf
  e}_h ^{m+1}$ in $l ^{\infty}( {\bf L}^2) $}

\begin{theorem}\label{dt1}
We assume  hypotheses of Theorem~\ref{v1} and the initial approximation
$$|{\bf e}_h^0 |
+ |k\,  \nabla e_{p,h}^0 | \le C\, h.
$$
 Then, the following
error estimates hold
\begin{equation}\label{dtr1}
\| \etilde _h   \|_{ l ^{\infty}( {\bf L}^2) \cap l ^{2}( {\bf
H}^1)}^2
 + \| {\bf e}_h^{m+1}  \| _{ l ^{\infty}( {\bf L}^2)}^2
+ \| k\,\nabla e_{p,h}^{m+1}  \| _{ l ^{\infty}( {\bf L}^2)}^2
  \le C\, h^{2},
\end{equation}
\begin{equation}\label{dtr1bis}
\| \etilde _h  -{\bf e}_h^m \| _{ l ^{2}( {\bf L}^2)}^2
 \le C\,k\,
  h^2  .
\end{equation}
\end{theorem}
\begin{remark}
By using the $O(k)$ accuracy for
the time discrete Algorithm~\ref{alg-time}, 
we
arrive at the following optimal order for the total error of the velocity:
$$ \|{\bf u}(t_{m+1}) -\utilde_h \|_{ l ^{\infty}( {\bf L}^2) \cap l ^{2}( {\bf
H}^1)} \le C \, (k+h).$$
\end{remark}
\noindent{\bf Proof}: By making $ \Big((E_1)_h^{m+1},2\, k\,
\emedio_h \Big)$ and using the equalities 
$$\Big(\nabla e_{p,h}^m, \ee_h
\Big)=0,$$ 
 $$2\, k\,
\Big(\nabla  e_{p,h}^m , \etilde_h \Big) = 2\Big(k\,\nabla
e_{p,h}^m, k\, \nabla (e_{p,h}^{m+1} -e_{p,h}^m) \Big) = | k\,
\nabla e_{p,h}^{m+1} |^2 - |k\, \nabla e_{p,h}^m | ^2- |k\, \nabla
(e_{p,h}^{m+1} -e_{p,h}^m) |^2,$$
  and the
$L^2$-orthogonality property
\begin{equation}\label{est-E2}
|\etilde_h |^2= |\ee_h|^2+ |k\, \nabla ( e_{p,h}^{m+1} -e_{p,h}^m) |^2,
\end{equation}
we arrive at
\begin{equation}\label{est-E1}
\begin{array}{l}
  |\ee_h |^2 -|{\bf e}_h^m|^2 + |\emedio_h - {\bf e}_h^m|^2+
2\, k\, \| \etilde_h  \|^2
 +  | k\, \nabla e_{p,h}^{m+1}
|^2 - |k\, \nabla e_{p,h}^m | ^2
  \\
   = -2\, k\, \Big(e_i (\delta_t \utilde) ,\emedio_h\Big) -2\,
k\,\Big(\nabla \emedio_i, \nabla\emedio_h \Big)-2\, k\, \Big(
\nabla (2\,e_{p,i}^m - e_{p,i}^{m-1}), \etilde_h \Big)\\
+ 2\, k\, c\Big(\widetilde{\bf e}_h^m, \utilde, \etilde_h \Big)
+2\,k\, c\Big(\widetilde{\bf e}_i^m , \utilde ,   \etilde_h\Big )
-2\,k\,c\Big( \widetilde{\bf u}_h^m , \etilde_h,\emedio_h \Big) \\
  -2\,k\, c\Big(\widetilde{\bf u}_h^m ,
\emedio _i, \emedio_h\Big )
 :=\sum_{i=1}^7 I_i
\end{array}
\end{equation}

We bound the RHS of (\ref{est-E1}) as follows (using  Remark \ref{reg-fuerte}):
$$
I_1 \le \varepsilon \, k\, | \etilde_h | ^2 + C \, k
|e_i(\delta_t \utilde) |^2 \le \varepsilon \, k\, \|
\etilde_h \| ^2 + C\,h^2 \, k \|\delta_t \utilde \|^2
$$
  $$
 I_2 \le \varepsilon \, k\, \|
\etilde_h \| ^2 + C\,h^2 \, k\, \|\utilde \|_{H^2}^2\le \varepsilon \, k\, \|
\etilde_h \| ^2 + C\,k\, h^2 
$$
$$
I_3=2\, k\, \Big(  (2\,e_{p,i}^m - e_{p,i}^{m-1}),
\nabla\cdot\etilde_h\Big) \le \varepsilon \, k\, \| \etilde_h \|
^2 +C\, k\, h^2 (\|p^m\|^2 + \| p^{m-1}\|^2) \le \varepsilon \, k\, \| \etilde_h \|
^2 +C\, k\, h^2
$$

With respect to the nonlinear terms,
\begin{eqnarray*}
I_4=2\, k \, c\Big(\widetilde{\bf e}_h^m , \utilde , \etilde_h
\Big)\le C\, k |\widetilde{\bf e}_h^m | \, \| \utilde  \|
_{W^{1,3} \cap L^{\infty}} \|  \etilde_h \| \le \varepsilon \, k\,
\| \etilde \| ^2 + C\, k\, | \widetilde{\bf e}_h^m | ^2
\end{eqnarray*}
$$ \le \varepsilon \, k\,
\| \etilde \| ^2 +C\, k\, \Big ( |{\bf e}_h^m |^2 + 2| k \, \nabla
e_{p,h}^{m}|^2+ 2|k \, \nabla e_{p,h}^{m-1}| ^2 \Big )
$$
(here,  (\ref{est-E2}) has been used),
\begin{eqnarray*}
I_5 =2\, k \, c\Big(\widetilde{\bf e}_i^m , \utilde , \etilde_h
\Big)\le \varepsilon \, k\, \| \etilde \| ^2 + C\, k\, |
\widetilde{\bf e}_i^m | ^2 \le \varepsilon \, k \, \| \etilde \|
^2 + C\, h^4  \, k\, \|\widetilde{\bf
  u}^m \| _{H^2}^2 
\end{eqnarray*}
$$ \le \varepsilon \, k \, \| \etilde \|
^2 + C\,k\,  h^4 ,
$$
\begin{eqnarray*}
I_6 =2\, k \, c\Big(\widetilde{\bf u}_h^m , \etilde_h, \etilde_h
\Big)=0,
\end{eqnarray*}
\begin{eqnarray*}
 I_7= 2\,k\, c\Big(\widetilde{\bf u}_h^m , \emedio _i, \emedio_h\Big) &\le&  C\,k
  \,\| \widetilde{\bf u}_h^m \| ^2 \, \| \emedio_i \|_{L^3} ^2+
\varepsilon \,k \| \emedio_h \| ^2 \\
   &\le &  C\,k \,\| \widetilde  {\bf u}_h^m \| ^2 \,
 | \emedio_i |  \,\| \emedio _i \|+
\varepsilon \,k \| \emedio_h \| ^2 \\
   &\le &  C\,k \,h^3 \, \| \widetilde{\bf u}_h^m \|^2 \,
   \,\| \umedio \|^2_{H^2} +
\varepsilon \,k \| \emedio_h \| ^2
\\
   &\le &  C\,k \,h^3 \, \| \widetilde{\bf u}_h^m \|^2 + \varepsilon \,k \| \emedio_h \| ^2
\end{eqnarray*}

Then, using these bounds in (\ref{est-E1}) we obtain
\begin{equation}\label{sumafinal}
\begin{array}{l}
    |{\bf e}_h^{m+1} |^2 -|{\bf e}_h^m|^2 + | k\, \nabla
e_{p,h}^{m+1} |^2 - |k\, \nabla e_{p,h}^m | ^2+ |\emedio_h - {\bf e}_h^m|^2+
 k\, \| \etilde_h  \|^2 \\
 \le C\, k\, \Big (|{\bf e}_h^m |^2 + | k \, \nabla
e_{p,h}^{m}|^2+ | k \, \nabla e_{p,h}^{m-1}| ^2\Big ) 
+C\, k\, h^2 + C\,k \,h^3
\, \| \widetilde{\bf u}_h^m \|^2 +C\, k\, h^2\, \|\delta_t \utilde\| ^2.
\end{array}
\end{equation}

Finally, by adding (\ref{sumafinal}) from
$m=0$ to $r$ (with any $r<M$), and using that $k\sum \| \widetilde{\bf
u}_h^m \|^2\le C$ and Theorem~\ref{lv12},
 the discrete Gromwall's
Lemma yields to
\begin{eqnarray*}
 |{\bf e}_h^{r+1} |^2  +| k\, \nabla
e_{p,h}^{r+1} |^2+ \sum_{m=0}^r  |\emedio_h - {\bf e}_h^m|^2+
 k\, \sum_{m=0}^r \| \etilde_h  \|^2  \le C\Big(|{\bf e}_h^0 |^2
+| k \,
 \nabla  e_{p,h}^0 | ^2 + h^2\Big)
\end{eqnarray*}
hence the estimates (\ref{dtr1})-(\ref{dtr1bis}) hold. \cqfd


Theorem~\ref{dt1} and  the inverse inequality $ \| {\bf u}_h \| \le C\,h^{-1} |{\bf u}_h   | $ for each ${\bf u}_h \in {\bf Y}_h$, imply the uniform estimate
\begin{equation}\label{uniform-estimate-H1}
 \| \emedio_h \| \le \frac{C}{h} \, | \emedio_h | \le C.
\end{equation}

\subsection{$O(h)$  for $\deltaee_h$ in $l ^{\infty}( {\bf L}^2)$,
 $\delta _t \emedio_h $ in $l ^{\infty}( {\bf L}^2) \cap l^{2}( {\bf H}^1)$
 and $(\etilde_h , e_{p,d}^{m+1})$ in $l ^{\infty}( {\bf H}^1 \times L^{2}) $ }

By making  $\delta_t (E_1)^{m+1}_h$ and $\delta_t (E_2)^{m+1}_h $,
one arrives at ($\forall\, m\ge 1 $):
$$
\frac{1}{k}  \Big(\delta_t \etilde_d -\delta_t {\bf e}_d^{m}, {\bf
v}_h  \Big)
 +  \,  \Big(\nabla   \delta_t \etilde_d ,\nabla\, {\bf v}_h
 \Big)-\Big(\delta_t \nabla
 e_{p,d}^{m} ,  {\bf v}_h \Big)
=  \delta_t  \,{\bf NL}_h^{m+1} ( {\bf v}_h )  \q \forall\, {\bf
v}_h\in {\bf Y}_h
$$
and
$$
\delta_t{\bf e}_d^{m+1}= \delta_t \emedio_d - k  \nabla (\delta_t
e_{p,d}^{m+1} -\delta_t e_{p,d}^m )
  $$
where
$$  \delta_t {\bf NL}_h^{m+1} ({\bf v}_h) = c\Big( \delta_t \widetilde{\bf
  e}_d^m  , \, \umedio, {\bf v}_h   \Big) + c\Big(\delta _t \widetilde{\bf
u}_h^m , \emedio_d , {\bf v}_h   \Big ) +  c\Big(\widetilde{\bf
e}_d^{m-1} ,  \delta _t \umedio , {\bf v}_h \Big) + c \Big(
\widetilde{\bf u}_h^{m-1} , \delta_t \emedio_d  , {\bf v}_h \Big).
$$
On the other hand, the following $L^2$-orthogonality property holds:
  $$
\Big( \delta_t \ee_d , \nabla q_h\Big)=0,\quad \forall\, q_h  \in
Q_h.
  $$
Consequently, for each ${\bf v}_h\in {\bf Y}_h$, one has
$$
 \left \{ \begin{array}{l}
\dis\frac{1}{k}  \Big(\delta_t \etilde_h -\delta_t {\bf e}_h^{m},
{\bf v}_h \Big )
 +  \, \Big (\nabla   \delta_t \etilde_h ,\nabla\, {\bf v}_h \Big )+
\Big (\nabla \delta_t
 e_{p,h}^{m} ,   {\bf v}_h\Big )
=  \delta_t  \,{\bf NL}_h^{m+1} ( {\bf v}_h )
\\
\quad  -\Big({\bf e}_i(\delta_t \delta_t \utilde),{\bf v}_h \Big) - \Big(\nabla \delta_t
\etilde_i, \nabla {\bf v}_h \Big) -\Big( \nabla (2 \, \delta_t
e_{p,i}^m - \delta_t e_{p,i}^{m-1} ), {\bf v}_h \Big) ,\end{array} \right.
\leqno{(D_1)^{m+1}_h}
$$
$$
 \delta_t \ee_h  = \delta_t \emedio_h - k
 \nabla (\delta_t e_{p,h}^{m+1} -\delta_t e_{p,h}^m ),
\leqno{(D_2)^{m+1}_h}
$$
and the following discrete $L^2$-orthogonality property:
\begin{equation}\label{discrete-orthog}
\Big( \delta_t \ee_h , \nabla q_h \Big)=0,\quad \forall\, q_h  \in Q_h.
\end{equation}
In the last two equalities, some properties of  the interpolation operators have been used.

\begin{theorem}\label{dt2} Under the hypotheses of Theorems
  \ref{lv12} and \ref{dt1}, assuming the following approximation for the first step of Algorithm~\ref{algo-fully}
 \begin{equation}\label{first-step}
 | \delta_t {\bf e}_h^{1} | + | k \,
 \nabla \delta_t e_{p,h}^1 | \le C\, h , \quad  k\, \|  \delta_t \widetilde{\bf e}_h^{1} \|^2 \le C\, h^2
\end{equation}
  then
\begin{equation}\label{error-delta_t-space}
\| \deltaee_h \|_{l^{\infty}({\bf L}^2)} +\| \delta_t \etilde _h
\|_{l^{\infty}({\bf L}^2)\cap l^2({\bf H}^1)} + \|k\,
\delta_t\nabla e_{p,h}^{m+1} \|_{l^{\infty}({\bf L}^2)} \le C\, h.
\end{equation}
\end{theorem}
\noindent {\bf Proof:} Since the initial estimate   $| \delta_t
{\bf e}_h^{1} | + | k \,
 \nabla \delta_t e_{p,h}^1 | \le C\, h $
is assumed, it suffices to prove \eqref{error-delta_t-space} for each $m\ge 1$.

 By adding   $(D_1)^{m+1}_h$ multiplied by $2\, k\, \delta_t \emedio_h  \in {\bf Y}_h$, where the pressure term is  writing as
\begin{eqnarray*}
  2\, k\, \Big(\nabla
\delta_t e_{p,h}^m , \delta_t \etilde_h \Big)&=&  2\, k\,
\Big(\nabla \delta_t e_{p,h}^m, k\, \nabla (\delta_t e_{p,h}^{m+1}
-\delta_t
e_{p,h}^m) \Big)  \\
   &=&  | k\, \nabla \delta_t e_{p,h}^{m+1} |^2 - |k\,
\nabla \delta_t e_{p,h}^m | ^2- |k\, \nabla (\delta_t
e_{p,h}^{m+1} -\delta_t e_{p,h}^m) |^2
\end{eqnarray*}
\Big (here $\Big (\nabla \delta_t e_{p,h}^m, \delta_t \ee_h \Big
)=0$ has been used\Big), and  the equality 
$$
|\delta_t \etilde_h |^2= |\delta_t \ee_h|^2+ |k\, \nabla (\delta_t
e_{p,h}^{m+1} -\delta_t e_{p,h}^m) |^2$$
(which is deduced from $(D_2)_h^{m+1}$ and the discrete 
$L^2$-orthogonality \eqref{discrete-orthog}),  one has
 \begin{equation}\label{dte1}
   \begin{array}{l}
     |  \delta_t \ee_h | ^2 -
|  \delta_t  {\bf e}_h^m | ^2 + |  \delta_t \emedio_h -\delta_t
{\bf e}_h^m | ^2 + 2  \, k \| \delta_t  \emedio_h \| ^2
+ | k\, \nabla \, \delta_t e_{p,h}^{m+1} |^2
\\
- | k\, \nabla \, \delta_t
e_{p,h}^m | ^2 
  = -2\, k\, \Big({\bf e}_i (\delta_t \delta_t \utilde),  \delta_t
  \emedio_h  \Big)
 - 2\, k\,\Big(\nabla
\delta_t \etilde_i, \nabla \delta_t \emedio_h\Big )
\\ -2\, k \,\Big(
\nabla (2 \, \delta_t e_{p,i}^m - \delta_t e_{p,i}^{m-1} ),
\delta_t \emedio_h \Big)
 +
  2\, \delta_t  {\bf NL}_h^{m+1}  ( \delta_t \emedio_h  )
:= I_1 +I_2 + I_3 +I_4.
 \end{array}
\end{equation}
We bound the RHS of \eqref{dte1} as:
$$I_1 \le \varepsilon \, k\, \| \delta_t \etilde_h \|^2 + C\, k \,
h^2 \, | \delta_t  \delta_t \utilde |^2$$
(here the hypothesis \eqref{H-1approx} on the $O(h)$-approximation of $I_h$ in the ${\bf H}^{-1}$-norm has been used),
$$I_2 \le  \varepsilon \, k\, \| \delta_t \etilde_h \|^2 + C\, k \,h^2
\| \delta_t  \utilde \| _{{\bf H}^2}^2 $$
$$ I_3 \le \varepsilon \, k\, \| \delta_t \etilde_h \|^2 + C\, k \,
h^2 (\| \delta_t p^m \|^2 + \| \delta_t p^{m-1} \|^2 ).
$$
The nonlinear terms, for $m\ge 1$, are treated as follows:
\begin{eqnarray*}
I_4 &= & 2\, k\, c\Big( \delta_t \widetilde{\bf e}_d^m  , \,
\umedio, \, \delta_t \emedio_h \Big) + 2\, k \, c\Big(\delta _t
\widetilde{\bf u}_h^m , \emedio_d, \,\delta_t \emedio_h  \Big)
\\
&+& 2\,k \, c\Big(\widetilde{\bf e}_d^{m-1} ,  \delta _t \umedio ,
  \, \delta_t \emedio_h \Big)
+ 2\, k\,c\, \Big( \widetilde{\bf u}_h^{m-1} , \delta_t \emedio_d
,  \,  \delta_t \emedio_h   \Big )
 :=
\sum_{i=1}^4 J_i
\end{eqnarray*}
We bound each  $J_i$-term as follows:
$$
J_1 = 2\, k \, c\Big( \delta_t \widetilde{\bf e}_h^m  , \,
\umedio, \, \delta_t \emedio_h  \Big)+ 2\, k\, c\Big( \delta_t
\widetilde{\bf e}_i^m , \, \umedio, \, \delta_t \emedio_h
\Big):=J_{11}+J_{12}
$$
$$J_{11}
\le \varepsilon \, k \| \delta_t \emedio_h \|^2 + C\, k \, \| \umedio
\|_{W^{1,3}\cap L^{\infty}}^2 \, | \delta_t  \widetilde{\bf e}_h^m |^2
$$
$$ 
\le \varepsilon \, k \, \| \delta_t \emedio_h \|^2 +C\, k\, \,\Big ( 
|\delta_t {\bf e}_h^m |^2 +2 (|k\, \nabla \delta_t e_p^m| ^2+ | k\,
\nabla \delta_t e_p ^{m-1} | ^2 )
\Big ),
$$
$$
J_{12}  \le \varepsilon \, k \| \delta_t  \etilde_h \|^2 +
C\, k \, \| \umedio\|^2 \, | \delta_t \widetilde{\bf e}_i^m  |^2 \le
\varepsilon \, k \| \delta_t  \widetilde{\bf e}_h ^{m+1} \|^2 + C\, k\, h^2 \,
\|\delta_t \widetilde{\bf u}^m \|^2,
$$

$$ J_2=  2\, k \, c\Big(\delta _t
\widetilde{\bf e}_h^m , \emedio_d, \,\delta_t \emedio_h  \Big)+   2\, k \, c\Big(\delta _t
I_h \, \widetilde{\bf u}^m , \emedio_h +\emedio_i, \,\delta_t
\emedio_h  \Big)$$
$$  \le \varepsilon \, k\, \Big ( \| \delta_t \widetilde {\bf e}_h^{m+1} \|
^2 +\| \delta_t \widetilde {\bf e}_h^{m} \| ^2 \Big )  + C\, k \, \| \etilde _d
\|^2  \, | \delta_t \widetilde {\bf e}_h^m |
^2 + C\, k\, \|I_h \, \delta _t
 \widetilde{\bf u}^m \|^2 \,  \Big (  \|
 \widetilde {\bf e}_i ^m \| ^2 + \|
 \widetilde {\bf e}_h ^m \| ^2\Big )
$$
$$  \le \varepsilon \, k\, \Big ( \| \delta_t \widetilde {\bf e}_h^{m+1} \|
^2 +\| \delta_t \widetilde {\bf e}_h^{m} \| ^2 \Big ) + C\, k \, | \delta_t \widetilde {\bf e}_h^m |
^2 + C\, k\, \|\delta _t
 \widetilde{\bf u}^m \|^2 \,  \Big ( h^2 \|
 \widetilde {\bf u} ^m \|_{{\bf H}^2} ^2 + \|
 \widetilde {\bf e}_h ^m \| ^2\Big )
$$
$$ \le \varepsilon \, k\, \Big ( \| \delta_t \widetilde {\bf e}_h^{m+1} \|
^2 +\| \delta_t \widetilde {\bf e}_h^{m} \| ^2 \Big ) 
+C\, k\, \,\Big ( 
|\delta_t {\bf e}_h^m |^2 +2 (|k\, \nabla \delta_t e_p^m| ^2+ | k\,
\nabla \delta_t e_p ^{m-1} | ^2 )\Big)+C\,k\, h^2+C\, k\,  \|
 \widetilde {\bf e}_h ^m \| ^2
$$
(in the last inequality we use $ \|
\etilde_d\| \le C$, due to (\ref{uniform-estimate-H1})  and
$\|\widetilde{\bf u}^{m}\|_{H^2} \le C$),
\medskip
$$J_3 = 2\,k \, c\Big(\widetilde{\bf e}_h^{m-1} , \delta _t \umedio ,
  \, \delta_t  \emedio_h  \Big)+ 2\,k \, c\Big(\widetilde{\bf e}_i^{m-1} , \delta _t \umedio ,
  \, \delta_t \emedio_h  \Big):=J_{31}+J_{32}
$$
$$J_{31}\le  C\,k \, \| \widetilde{\bf e}_h^{m-1} \| \, \|   \deltaumedio
\|_{L^3}  \,\| \delta_t  \emedio_h   \|
\le \varepsilon \, k\,  \|\deltaemedio_h \|^2 +C\, k
\, \| \widetilde{\bf e}_h^{m-1} \| ^2,$$
$$J_{32} \le \varepsilon \, k\, \| \delta_t  \emedio_h  \|^2 + C\, k\, \| \widetilde{\bf e}_i^{m-1} \| ^2
\le \varepsilon \, k\, \|\deltaemedio_h \|^2 +C\, k \| \delta_t
\utilde \|_{H^2} ^2 \, \| \widetilde{\bf e}_i ^{m-1} \| ^2,
$$

$$
J_4= 2\, k\,c\Big( \widetilde{\bf u}_h^{m-1} , \delta_t \emedio_h
,  \, \delta_t  \emedio_h   \Big)+ 2\, k\,c\Big( \widetilde{\bf
u}_h^{m-1} , \delta_t \emedio_i ,
  \, \delta_t \emedio_h  \Big ):= J_{41}+ J_{42}
$$
$$J_{41}=0,$$
$$
 J_{42} \le C\, k\, \, \|\widetilde{\bf u}_h ^{m-1} \|\, \| \delta_t \emedio_i
\|_{L^3} \,  \| \delta_t \emedio_h  \|
   \le \varepsilon\, k\, \|\delta_t \emedio_h \| ^2
+ C\, k\,\| \widetilde {\bf u}_h^{m-1} \| ^2\, h^3  \| \delta_t \umedio \|_{H^2}^2
$$
(here,  we use  
$ \| \delta_t
\emedio_i \|_{L^3}\le C \, | \delta_t  \widetilde {\bf e}_i^{m+1}
| ^{1/2} \, \|  \delta_t \emedio_i \| ^{1/2} \le C\, h^{3/2} \|
\delta_t \umedio \|_{H^2}
$). 

By applying these estimates in (\ref{dte1}) for a small enough $\varepsilon$,  we obtain
\begin{eqnarray*}
   &&  |\delta_t {\bf e}_h^{m+1} |^2 -|\delta_t {\bf e}_h^m|^2 + |\delta_t
   \emedio_h - \delta_t {\bf e}_h^m|^2+
 k\, \| \delta_t \etilde_h  \|^2 +| k\, \nabla
\delta_t e_{p,h}^{m+1} |^2 - |k\,  \nabla\, \delta_t e_{p,h}^m | ^2
 \\
&&\le\ C\, k\, \,\Big ( 
|\delta_t {\bf e}_h^m |^2 +2 (|k\, \nabla \delta_t e_p^m| ^2+ | k\,
\nabla \delta_t e_p ^{m-1} | ^2 )
\Big )+C\, k\, h^2 + \frac{k}{2} \, \| \delta_t
 \widetilde{\bf e}^m \| ^2 \\
&& + C\, k \,
h^2 \,\Big ( | \delta_t  \delta_t \utilde |^2 + 
\| \delta_t  \utilde \| _{H^2}^2 \Big )+C\, k\,  \|
 \widetilde {\bf e}_h ^m \| ^2 .
\end{eqnarray*}
Therefore, by adding from $m=1$ to $r$
(with any $r<M$), taking into account (\ref{reg-delta-delta}), 
Lemma \ref{reg-delta-h2} and Theorem \ref{dt1}, the discrete Gromwall's Lemma can be applied, yielding to 
\begin{eqnarray*}
&& |\delta_t {\bf e}_h^{r+1} |^2 +| k\, \nabla
\delta_t e_{p,h}^{r+1} |^2+ \sum_{m=1}^r  |\delta_t \emedio_h -
\delta_t {\bf e}_h^m|^2+
 \frac{k}2\, \sum_{m=1}^r \| \delta_t \etilde_h  \|^2
 \\ && \le
 C\Big(|\delta_t {\bf e}_h^1 |^2 + \frac{k}{2} \, \| \delta_t
 \widetilde{\bf e}^1 \| ^2
 + | k \, \nabla \delta_t e_{p,h}^1 | ^2 +    h^2\Big),\end{eqnarray*}
hence  (\ref{error-delta_t-space}) holds by using the hypotheses on the first step \eqref{first-step}.
\cqfd

%

\begin{corollary} \label{dt5}
Assuming hypotheses of  Theorem \ref{dt2}, the following error
estimates hold
$$
\| \etilde_h \| _{l^{\infty}({H^1})} \le C\, h
\quad \hbox{and} \quad 
  \| e_{p,h}^{m+1} \| _{l^{\infty}(L^2)} \le C\,h  .
$$
\end{corollary}

\noindent {\bf Proof:}  We divide the proof into three steps:

\noindent{\bf Step 1}. To obtain 
\begin{equation}\label{est-L2-presion}
 \| e_{p,h}^{m+1} \|_{l^2(L^2)} \le C\, h.
\end{equation}
 Arguing as in the time discrete Algorithm~\ref{alg-time}, from
the discrete
 inf-sup condition applied to $ (E_3 )_h ^{m+1}$ and the estimates 
$\| \etilde_h \|_{l^2 (H^1)} \le C\, h$ and $\| \delta_t \ee_h
\|_{l^{\infty}(L^2)} \le C\, h $, we have \eqref{est-L2-presion}.

\noindent{\bf Step 2}. To prove \begin{equation}\label{reg-etilde-Linfty}
\|\etilde_h  \|_{l^{\infty} (H^1)} \le
  C\, h .
\end{equation}
By multiplying  $(E_3)_h^{m+1}$ by $2\,k\,
\delta _t \etilde_h $:
$$ \, |\nabla \etilde_h |^2
 - | \nabla \widetilde{\bf e}_h^{m}|^2 +  |\nabla  \etilde_h -
 \nabla \widetilde{\bf e}_h^{m}|^2 =
-2\, k \Big(\nabla e_{p,h} ^{m+1} +  \delta_t {\bf e}_h^{m+1} + \delta_t {\bf e}_i^{m+1}  ,  \delta _t \etilde_h   \Big)
$$
$$-2\,
k\, \Big(\nabla  \etilde _i , \nabla\delta_t \etilde_h\Big) -2\, k \,(\nabla (2\,
e_{p,i}^m -e_{p,i}^{m-1} ) , \delta_t \etilde _h )+ 2\, k\,
{\bf NL}_h^{m+1}( \delta _t \etilde_h  ).
$$
Then, we obtain
$$  |\nabla \etilde_h|^2 -
| \nabla \widetilde{\bf e}_h^{m}|^2 +  |\nabla  \etilde_h -
 \nabla \widetilde{\bf e}_h^{m}|^2 \le 2\, k \,
|e_{p,h} ^{m+1} |^2
$$
$$+ C\, k \,  |\nabla \cdot \delta_t \etilde_h|^2 +C\,k\,
|\delta_t {\bf e}_h ^{m+1} |^2 +C\,k\, |\delta_t\widetilde{\bf
e}_h ^{m+1} |^2
$$
$$ +2\, k\, |e_i (\delta_t \uu)| ^2 + 2\, k\, \|\etilde_i\|^2
  +C\, k \, (| e_{p,i} ^{m} |^2 + |e_{p,i} ^{m-1} |^2)
  +2\, k\,  {\bf NL}_h^{m+1}(  \delta_t \etilde_h)
$$
$$   \le 2\, k \, |e_{p,h} ^{m+1} |^2  + C\, k \,  |\nabla \cdot \delta_t \etilde_h|^2 +C\,k\,
|\delta_t {\bf e}_h ^{m+1} |^2 +C\,k\, |\delta_t \widetilde{\bf
e}_h ^{m+1} |^2
$$
$$  + C\, k\, h^2 \|\delta_t \uu\|^2 +
 C\, k\, h^2 \|\utilde \|_{H^2} ^2
  + C\,k\, h^2  (\|p^{m} \|^2+\|p^{m-1} \|^2)
  +2\, k\, {\bf NL}_h^{m+1}(  \delta_t \etilde_h) .
$$
Taking
into account (\ref{uniform-estimate-H1}), we bound  the last term of the RHS as follows, 
$$2\, k\, {\bf NL}_h^{m+1}(  \delta_t \etilde_h)\le \varepsilon
  \, k
  \|  \delta_t \etilde_h \|^2 + C\, k\, \| \etilde_h \| ^2+ C\, k \, \|
  \widetilde{\bf e}_h^m \|^2 + C\, k\, \| \etilde_i \| ^2+ C\, k \|
  \widetilde{\bf e}_i^m \|^2
$$
$$ \le  \varepsilon \, k
  \|  \delta_t \etilde_h \|^2 + C\, k\, \| \etilde_h \| ^2+ C\, k \,\|
  \widetilde{\bf e}_h^m \|^2 +C\, k\, h^2 
$$
hence, we arrive at
$$
  |\nabla \etilde_h|^2 -| \nabla \widetilde {\bf e}_h^m |^2 + |\nabla  \etilde_h -
 \nabla \widetilde{\bf e}_h^{m}|^2 \le
  k \,
|e_{p,h} ^{m+1} |^2
 +k  |\nabla \cdot \delta_t \etilde_h|^2
  +C\, k\,
|\delta_t {\bf e} ^{m+1}_h |^2 
$$
$$+ k | \nabla \, \delta_t \etilde_h
| ^2  + C\, k\, | \delta _t  \etilde_h  | ^2 +  C\, k\, \|
\etilde_h \| ^2+ C\, k \|
  \widetilde{\bf e}_h^m \|^2 + C\, k\, h^2   + C\, k\, h^2 \|\delta_t \uu\|^2  $$
Adding from $m=0$ to $r$,
$$ |\nabla \widetilde{\bf e}_h^{r+1}|^2
\le |\nabla \widetilde{\bf e}_h^0 |^2+ C\, k \sum_{m=0} ^r
|e_{p,h} ^{m+1} |^2 +C\, k \sum_{m=0} ^r |\nabla \, \delta_t
\etilde_h|^2 + C\, k\, \sum_{m=0} ^r |\delta_t {\bf e} ^{m+1}_h
|^2 + C\, h^2  . $$ Then, by applying (\ref{est-L2-presion}) and  the estimates  obtained  in Theorems \ref{dt1} and \ref{dt2}, we obtain \eqref{reg-etilde-Linfty}.
 
\noindent {\bf Step 3}. To obtain $\| e_{p,h}^{m+1} \|_{l^{\infty}(L^2)} \le C\, h$.

Finally, by using again  the discrete inf-sup  condition \eqref{cinfsup} and
  taking into account (\ref{reg-etilde-Linfty}), one has $\| e_{p,h}^{m+1} \|_{l^{\infty}(L^2)} \le C\, h$ and the proof is finished.
 \cqfd

\begin{remark} \label{re:total-order}
By combining  Theorem  \ref{cor1} and Corollary \ref{dt5},  the
following error estimate for the total  error holds
$$
\| {\bf u}(t_{m+1}) - \utilde_h \| _{l^{\infty}(H^1)}+ \| p(t_{m+1}) - p_h^{m+1} \|_{l^{\infty}(L^2)}
\leq C\,(k+ h).
$$

\end{remark}
\section{Numerical Simulations}
We  consider the FE approximation $P_2 \times P_1$  related to a structured mesh of the domain $\Omega =(0,1)^2 \subset \R^2  $. 

 The numerical results have been obtained using the software  FreeFem++ (\cite{ff++,hecht}), and show 
 first order accurate in time for velocity and pressure of the segregated version of  Algorithm~\ref{algo-fully} given in Remark~\ref{segr-Al-2}. These results are agree to  Remark~\ref{re:total-order}.

In fact, we present some numerical error orders  in time for velocity ${\bf u}=(u_1,u_2)$ and pressure $p$ using the following exact solution for $(P)$:
$$ {\bf u}=e^{-t} \Big ( \begin{array}{c}
                 (\cos ( 2 \, \pi \, x) -1) \, \sin (2 \, \pi \, y)\\
 - (\cos (2 \, \pi \, y) -1) \, \sin( 2 \, \pi \, x)
 \end{array}
\Big)
\quad\hbox{and}\quad p= 2\, \pi \, e^{-t} \, (\sin (2 \, \pi \, x)
+\sin ( 2 \, \pi \, y)).
$$
We take $\nu=1$ and adjust the force ${\bf f}$ to enforce this exact solution.

Note that $\nabla\cdot{\bf u}=0$ in $\Omega$, ${\bf u}|_{\partial\Omega}=0$ and $\int_\Omega p =0$. On the other hand, we have choice this regular exact solution such that $\nabla p\cdot {\bf n}\not=0$ on the boundary $\partial\Omega$, in order to measure the effect of the numerical boundary condition $\nabla (p^{n+1}-p^n)\cdot {\bf n}=0$ on  $\partial\Omega$. 
 We approach numerically the order in time for the
 segregated version of Algorithm~\ref{algo-fully} given in Remark~\ref{segr-Al-2}, comparing to  other  current  first order  splitting schemes like, rotational pressure-correction, consistent splitting  and   penalty-projection schemes, also implemented  in  they segregated form.

Some numerical analysis results and   computational simulations  can be seen in
\cite{guermond-shen} and  \cite{guermond-minev-shen} for the
rotational pressure-correction projection scheme, in
\cite{guermond-shen-consistent}, 
\cite{guermond-minev-shen} and \cite{shen-yang} for the  consistent
splitting  scheme and in  \cite{jobelin}, \cite{angot-1} and 
\cite{angot-2} for the penalty-projection scheme. 

Concretely, let $ {\bf u}_h^m\in {\bf Y}_h$, $\Pi_h( \nabla \cdot {\bf u}^m)\in Q_h$ and $p_h^{m-1},p_h^m\in Q_h$ be given,  where $\Pi_h$ is the  $L^2 (\Omega)$-projector operator onto the discrete pressure space $Q_h$, the implemented segregated  schemes   are:
\begin{itemize}
\item 
The {\it rotational pressure-correction} scheme:
\begin{description}
\item (a) Find ${\bf u}_h^{m+1} \in {\bf Y}_h $ such that  $ \forall \,{\bf v}_h \in
{\bf Y}_h $,
$$ \Big ( \frac{{\bf u}_h^{m+1} - {\bf u}_h^m}{k}, {\bf v}_h \Big )
+ c\Big( {\bf u}_h^{m}, {\bf u}_h^{m+1},  {\bf
v}_h\Big)
 +  \nu\Big(\nabla {\bf u}_h^{m+1},  \nabla  {\bf v}_h\Big) -
 \Big(  q^m_h ,   \nabla\cdot{\bf v}_h\Big)  =
\Big({\bf f}^{m+1},  {\bf v}_h\Big).
$$
where 
\begin{equation}\label{potent-aux}
q^m_h=2 p_h ^{m} -p_h ^{m-1}+\nu \,\Pi_h( \nabla \cdot {\bf u}^m)\in Q_h .
\end{equation}
\item (b) Compute $\Pi_h(\nabla \cdot {\bf u}_h^{m+1})$.
  \item (c) Find  $p_h^{m+1}\in Q_h$ such that
\begin{equation}\label{press-increm}
k \Big(\nabla\,(p_h^{m+1} -  p_h^m + \nu \, \Pi_h(\nabla \cdot {\bf u}_h^{m+1})) ,\nabla q_h \Big)= - \Big(\nabla \cdot
{\bf u}_h^{m+1},q_h \Big)  \quad \forall\, q_h\in Q_h
\end{equation}
\end{description}
\item
The {\it consistent splitting}  scheme:
\begin{description}
\item (a) Find ${\bf u}_h^{m+1} \in {\bf Y}_h $ such that $ \forall \,{\bf v}_h \in
{\bf Y}_h $,
 $$ \Big ( \frac{{\bf u}_h^{m+1}-{\bf u}_h^m}{k}, {\bf v}_h \Big ) 
 + c\Big( {\bf u}_h^{m}, {\bf u}_h^{m+1},  {\bf v}_h\Big)
 + \nu  \Big(\nabla  {\bf u}_h^{m+1},\nabla \, {\bf v}_h \Big) -\Big(p_h^m , \nabla \cdot {\bf v}_h \Big) = \Big({\bf f}^{m+1},{\bf v}_h \Big  )$$
 \item (b) Compute $\Pi_h(\nabla \cdot {\bf u}_h^{m+1}))$.
 \item (c) Find  $p_h^{m+1}\in Q_h$ such that
$$\Big(\nabla \, (p_h^{m+1} -  p_h^m + \nu \, \Pi_h(\nabla \cdot {\bf u}_h^{m+1})) , \nabla  q_h \Big)= \Big ( \frac{{\bf u}_h^{m+1}-{\bf u}_h^m}{k}, \nabla q_h\Big ) \quad \forall\,q_h \in Q_h.$$
\end{description}
\item
The  {\it penalty pressure-projection} scheme: 
\begin{description}
\item (a) Find ${\bf u}_h^{m+1} \in {\bf Y}_h $ such that,  $ \forall \,{\bf v}_h \in
{\bf Y}_h $,
$$ \Big ( \frac{{\bf u}_h^{m+1} - {\bf u}_h^m}{k}, {\bf v}_h \Big )
+ c\Big( {\bf u}_h^{m}, {\bf u}_h^{m+1},  {\bf v}_h\Big)
 +\nu  \Big(\nabla {\bf u}_h^{m+1},  \nabla {\bf v}_h\Big) +\nu    \Big(\nabla \cdot   {\bf u}_h^{m+1},  \nabla \cdot\, {\bf v}_h\Big) 
$$
$$ - \Big(  q^m_h ,  \nabla\cdot {\bf v}_h\Big)
 =\Big({\bf f}^{m+1},  {\bf v}_h\Big).
$$
where $q^m_h$ is given as in \eqref{potent-aux}. Note that this scheme is not fully segregated because it  couples the velocity components in the term $\nu    \Big(\nabla \cdot   {\bf u}_h^{m+1},  \nabla \cdot\, {\bf v}_h\Big) $.
\item (b) Compute $\Pi_h(\nabla \cdot {\bf u}_h^{m+1}))$.
\item (c) Find  $p^{m+1}\in Q_h$ solving the Poissson-Neumann problem 
\eqref{press-increm}.
\end{description}
\end{itemize}
We consider the structured mesh taking $70$  subintervals in $[0,1]$ (with  $h=0.0142857$).  In addition,  $k=0.2, \ 0.1, \ 0.05$ and
$0.025$ are considered corresponding to $10, 20, 40$ and $80$ time iterations in the  time interval $[0,2]$.  

The numerical results comparing the time accuracy can be seen  in Tables~\ref{fig:error-order-Incremental}, \ref{fig:error-order-Rotational}, \ref{fig:error-order-Consistent} and \ref{fig:error-order-penalty}, showing  a little better accuracy   in velocity and  pressure  for the
incremental scheme Algorithm~\ref{algo-fully}. Moreover,  
first  order accurate in time for velocity and pressure is observed for all previous schemes.   







\begin{table}
\begin{tabularx}{\textwidth}{|X|X|X|X|}
\hline
 $k$   & $0.2-0.1$ &  $0.1- 0.05$ & $0.05$-   $ 0.025$
\\ \hline\hline
$\|u_1\|_{l^{\infty}(L^2)} $ &  1.077  & 1.326  & 1.582
\\ \hline
$\|u_1\|_{l^{\infty}(H^1)}$ &  0.812  & 1.146  & 1.453 
\\ \hline
$\|u_2\|_{l^{\infty}( L^2)}$ & 1.095  & 1.352  & 1.585 
\\ \hline
$\|u_2\|_{l^{\infty}( H^1)}$ & 0.817  & 1.148  & 1.457 
\\ \hline
$\|p\|_{l^2(L^2)} $ & 0.877  & 1.282  & 1.535 
\\ \hline
$\|p\|_{l^{\infty}(L^2)} $ & 0.880  & 1.157  & 1.444
\\ \hline
 \end{tabularx}
\caption{\textit{} Error orders in time for Algorithm~\ref{algo-fully}}
 \label{fig:error-order-Incremental}
 \end{table}
%


\begin{table}
\begin{tabularx}{\textwidth}{|X|X|X|X|}
\hline
 $k$   & $0.2-0.1$ &  $0.1- 0.05$ & $0.05$-   $ 0.025$
\\ \hline\hline
$\|u_1\|_{l^{\infty}(L^2)} $ & 1.048  & 1.278  & 1.475 
\\ \hline
$\|u_1\|_{l^{\infty}(H^1)}$ & 0.955  & 1.150  & 1.290
\\ \hline
$\|u_2\|_{l^{\infty}( L^2)}$ &  1.105  & 1.314  & 1.511 
\\ \hline
$\|u_2\|_{l^{\infty}( H^1)}$ & 1.035  & 1.176  & 1.311 
\\ \hline
$\|p\|_{l^2(L^2)} $ & 1.241  & 1.436  & 1.490 
\\ \hline
$\|p\|_{l^{\infty}(L^2)} $ & 1.012  & 1.238  & 1.361 
\\ \hline
 \end{tabularx}
\caption{\textit{} Error orders in time for  Rotational Scheme }
 \label{fig:error-order-Rotational}
 \end{table}

\begin{table}
\begin{tabularx}{\textwidth}{|X|X|X|X|}
\hline
 $k$   & $0.2-0.1$ &  $0.1- 0.05$ & $0.05$-   $ 0.025$
\\ \hline\hline
$\|u_1\|_{l^{\infty}(L^2)} $ &  0.726  & 0.814  & 0.885 
\\ \hline
$\|u_1\|_{l^{\infty}(H^1)}$ & 0.715  & 0.813  & 0.885
\\ \hline
$\|u_2\|_{l^{\infty}( L^2)}$ & 0.764  & 0.843  & 0.905
\\ \hline
$\|u_2\|_{l^{\infty}( H^1)}$ &  0.775  & 0.841  & 0.908 
\\ \hline
$\|p\|_{l^2(L^2)} $ & 0.822  & 0.906  & 0.952 
\\ \hline
$\|p\|_{l^{\infty}(L^2)} $ & 0.700  & 0.792  & 0.868
\\ \hline
 \end{tabularx}
\caption{\textit{} Error orders in time for  Consistent Scheme }
 \label{fig:error-order-Consistent}
 \end{table}
\begin{table}
\begin{tabularx}{\textwidth}{|X|X|X|X|}
\hline
 $k$   & $0.2-0.1$ &  $0.1- 0.05$ & $0.05$-   $ 0.025$
\\ \hline\hline
$\|u_1\|_{l^{\infty}(L^2)} $ & 0.983  & 1.256  & 1.459 
\\ \hline
$\|u_1\|_{l^{\infty}(H^1)}$ & 0.903  & 1.120  & 1.266
\\ \hline
$\|u_2\|_{l^{\infty}( L^2)}$ & 1.012  & 1.265  & 1.484 
\\ \hline
$\|u_2\|_{l^{\infty}( H^1)}$ & 0.942  & 1.135  & 1.282
\\ \hline
$\|p\|_{l^2(L^2)} $ & 1.161  & 1.354  & 1.429
\\ \hline
$\|p\|_{l^{\infty}(L^2)} $ & 0.937  & 1.172  & 1.324 
\\ \hline
 \end{tabularx}
\caption{\textit{} Error orders in time for   Penalty-Projection Scheme }
 \label{fig:error-order-penalty}
 \end{table}

With respect to the computational cost, the CPU time needed  taking $k=0.025$ ($80$ time iterations)  is shown in
Table~\ref{comput-const}, showing a little 
 lower cost in the incremental scheme Algorithm~\ref{algo-fully}. Note that in this scheme  the  problem related to  the  $L^2 (\Omega)$-projector $\Pi_h$ has  not to be  computed.
\begin{table}
\begin{tabularx}{\textwidth}{|X|X|X|X|X|}
\hline
Scheme: & Algorithm~\ref{algo-fully} & Rotational   & Consistent & Penalty 
\\ \hline
CPU-time (s) &  \mbox{2067.45} &   \mbox{2113.22} &\mbox{2079.4}  &  \mbox{2147.9} 
\\ \hline
 \end{tabularx}
\caption{Computational cost}
 \label{comput-const}
 \end{table}
%
\section{Conclusions}

The optimal error estimates of order $O(k+h)$ for the velocity and pressure are deduced  for the first-order linear fully
discrete  segregated scheme based on an incremental pressure
projection method (Algorithm~\ref{algo-fully}) approaching the $3D$ Navier-Stokes problem. This convergence is unconditional, i.e.~without imposing constraints on mesh size $h$ or time step $k$.

Moreover, some numerical computations of the segregated version of Algorithm~\ref{algo-fully} agree the previous numerical analysis are provided. These simulations are also compared with the segregated versions of the 
rotational, consistent and penalty-projection schemes, obtaining    a little better  accuracy in time and lower computational cost  of  Algorithm~\ref{algo-fully}.

Finally, although this segregated scheme has the numerical boundary layer furnished by the artificial boundary condition $\nabla(p ^{m+1}-p^m)\cdot {\bf n}$ on $\partial\Omega$, this fact does not perturb the optimal convergence in the energy norms ${\bf H}^1(\Omega)\times L^2(\Omega)$ for the velocity and pressure, respectively.

\end{document}